\documentclass[14 pt, a4paper]{article}
\usepackage{latexsym}
\usepackage{floatpag}
\usepackage{multirow}
\usepackage{graphicx}
\usepackage{amsmath}
\usepackage{amssymb}
\usepackage{tikz}
\usepackage{caption}
\usepackage{tkz-euclide}
\usetikzlibrary{positioning}
\usepackage{centernot}
\usepackage{amsfonts}
\usepackage{changepage}
\usepackage{enumerate}
\usepackage[a4paper,left=2cm,right=2cm,top=2cm,bottom=2cm]{geometry}
\usepackage{diagbox}
\usepackage{authblk}
\usepackage{tabu}
\usepackage{amsthm}
\usepackage{hyperref}
\usepackage{float}
\usetikzlibrary{calc}
\usepackage{lipsum}

\makeatletter
\def\th@plain{%
  \thm@notefont{}
  \itshape 
}
\def\th@definition{%
  \thm@notefont{}
  \normalfont 
}
\makeatother

\newtheorem{theorem}{Theorem}
\newtheorem{col}{Corollary}

\newtheorem{lemma}{Lemma}

\newtheorem{obs}{Observation}
\newtheorem{conj}{Conjecture}

\newtheorem{q}{Question}
\newtheorem{probl}{Problem}
\newtheorem{constr}{Construction}
\theoremstyle{definition}
\newtheorem{definition}{Definition}
\newtheorem{rem}{Remark}
\theoremstyle{plain}

\DeclareMathOperator{\ex}{ex}
\DeclareMathOperator{\sat}{sat}
\DeclareMathOperator{\tsat}{tsat}

\usepackage{graphicx} 

\title{Twin-free $K_r$-saturated Graphs and Maximally Independent Sets in $K_3$-free Graphs}
\author{Asier Calbet\footnote{School of Mathematical Sciences, Queen Mary University of London, United Kingdom, \href{mailto:a.calbetripodas@qmul.ac.uk}{\nolinkurl{a.calbetripodas@qmul.ac.uk}}}}
\date{}

\begin{document}

\maketitle

\begin{abstract}

\noindent We say that two vertices are twins if they have the same neighbourhood and that a graph is $K_r$-saturated if it does not contain $K_r$ but adding any new edge to it creates a $K_r$. In 1964, Erd\H{o}s, Hajnal and Moon showed that $\sat(n,K_r)=(r-2)n+o(n)$ for $r \geq 3$, where $\sat(n,K_r)$ is the minimum number of edges in a $K_r$-saturated graph on $n$ vertices, and determined the unique extremal graph. This graph has many twins, leading us to define $\tsat(n,K_r)$ to be the minimum number of edges in a twin-free $K_r$-saturated graph on $n$ vertices. We show that $(5 +2/3)n + o(n) \leq \tsat(n,K_3) \leq 6n+o(n)$ and that $\left(r+2\right)n + o(n) \leq \tsat(n,K_r) \leq (r+3)n+o(n)$ for $r \geq 4$. We also consider a variant of this problem where we additionally require the graphs to have large minimum degree. Both of these problems turn out be intimately related to two other problems regarding maximally independent sets of a given size in $K_3$-free graphs and generalisations of these problems to $K_r$ with $r \geq 4$. The first problem is to maximise the number of maximally independent sets given the number of vertices and the second problem is to minimise the number of edges given the number of maximally independent sets. They are interesting in their own right.
\end{abstract}

\section{Introduction}

Given a graph $H$, we say that a graph $G$ is $H$-saturated if it is maximally $H$-free, meaning $G$ contains no copy of $H$ but every graph $G+e$ obtained by adding a new edge $e$ to $G$ contains $H$. The saturation problem is to determine, or at least estimate, the saturation number $\sat(n,H)$, defined (for graphs $H$ with at least one edge and integers $n \geq 0$) to be the minimum number of edges in an $H$-saturated graph $G$ on $n$ vertices. We are usually interested in fixed $H$ and large $n$. The saturation problem is dual to the Tur\'an forbidden subgraph problem of determining the extremal number $\ex(n,H)$, defined to be the maximum number of edges in an $H$-saturated graph $G$ on $n$ vertices. However, much less is known for the saturation problem than for the  Tur\'an problem. See \cite{Saturation Survey} for a survey of the saturation problem.  \\

For certain $H$ the saturation number is known exactly. In particular, Erd\H{o}s, Hajnal and Moon showed (Theorem 1 in \cite{Erdos Hajnal and Moon}) that for $r \geq 2$ and $n \geq r-2$, $\sat(n,K_r)=(r-2)n-\binom{r-1}{2}$ and that the unique extremal graph consists of a $K_{r-2}$ fully connected to an independent set of size $n-(r-2)$ (for $n<r$, $K_n$ is the unique $K_r$-saturated graph on $n$ vertices). Note that for large $n$ this graph contains many vertices of degree $r-2$. Moreover, this is the smallest possible degree of a vertex in a $K_r$-saturated graph on $n \geq r-1$ vertices. One might therefore ask what happens if we forbid vertices of degree $r-2$, and more generally, vertices of small degree. In \cite{Mine} (see the second paragraph of the introduction), this led us to define $\sat(n,K_r,t)$ for integers $r \geq 3$, $t \geq r-2$ and large enough integers $n$ to be the minimum number of edges in a $K_r$-saturated graph $G$ on $n$ vertices with $\delta(G) \geq t$ (more precisely, such graphs exist if and only if $n \geq (r-1)t / (r-2) $). This quantity (or rather the one obtained by replacing the condition $\delta(G) \geq t$ with $\delta(G)=t$) was first considered by Duffus and Hanson (second paragraph of the Introduction in \cite{Duffus and Hanson}). \\

One could also note that the extremal graph has $r-2$ conical vertices (vertices adjacent to all other vertices), leading one to define $\sat^{\Delta}(n,K_r)$ to be the minimum number of edges in a $K_r$-saturated graph $G$ on $n$ vertices with $\Delta(G) \leq \Delta$, and indeed this quantity has been studied (see section 2.1 in \cite{Saturation Survey}). In this paper, we instead note that for large $n$ the extremal graph has many pairs of twins (vertices with the same neighbourhood), leading us to define $\tsat(n,K_r)$ for integers $r \geq 3$ and large enough integers $n$ to be the minimum number of edges in a twin-free $K_r$-saturated graph $G$ on $n$ vertices. It is not immediately clear whether such graphs even exist, but we show that they do for fixed $r$ and large $n$. In fact, we completely determine for which $r$ and $n$ such graphs exist.

\begin{theorem}\label{existence}
Let $r \geq 3$ and $n \geq 0$ be integers. Then there exists a twin-free $K_r$-saturated graph on $n$ vertices unless $n=r$, $n=r+1$, $r=3$ and $n=6$, or $r=3$ and $n=7$.
\end{theorem}

\newtheorem*{repthex}{Theorem \ref{existence}}

Note the trivial lower bound $\tsat(n,K_r) \geq \sat(n,K_r)=(r-2)n+o(n)$. We improve upon this bound and obtain an upper bound. \newline

\begin{theorem}\label{twin free result}
We have the following estimates for $\tsat(n,K_r)$.
\begin{enumerate}
\item For all integers $r \geq 3$,
$$\left(r+2\right)n + o(n) \leq \tsat(n,K_r) \leq (r+3)n+o(n) \ .$$
\item In the special case $r=3$, we obtain a stronger lower bound.
$$\left(5 + \frac{2}{3}\right)n + o(n) \leq \tsat(n,K_3) \leq 6n+o(n) \ .$$
\end{enumerate}
\end{theorem}

\newtheorem*{repth12}{Theorem \ref{twin free result}}

The author conjectures that the upper bound is tight. \newline

One can also combine the twin-free and minimum degree conditions. The naive approach would be to define $\tsat(n,K_r,t)$ for integers $r \geq 3$, $t \geq r-2$ and large enough integers $n$ to be the minimum number of edges in a twin-free $K_r$-saturated graph $G$ on $n$ vertices with $\delta(G) \geq t$. We now explain why a more complicated definition is in fact more natural. \newline

A blow-up of a graph $H$ is a graph $G$ obtained by replacing each vertex of $H$ with some positive number of copies. Two copies in $G$ are adjacent if and only if the vertices in $H$ they are copies of are adjacent (in particular copies of the same vertex are non-adjacent). Note that any two copies in $G$ of a vertex in $H$ are twins. It is easy to check that $G$ is $K_r$-saturated if and only if $H$ is $K_r$-saturated, unless $H$ is a clique on at most $r-2$ vertices and $G$ is a proper blow-up of $H$. An equivalent definition of $\tsat(n,K_r)$ is the minimum number of edges in a $K_r$-saturated graph $G$ on $n$ vertices that is minimal in the sense that it cannot be obtained from a smaller $K_r$-saturated graph $H$ in this way. \newline

Note that $ \delta(G) \geq \delta(H)$. Hence, given integers $r \geq 3$ and $t \geq r-2$ and a $K_r$-saturated graph $H$ with $\delta(H) \geq t$, we can generate other $K_r$-saturated graphs $G$ with $\delta(G) \geq t$ by taking blow-ups of $H$ (the conditions $\delta(H) \geq t \geq r-2$ rule out the trivial exceptions). In \cite{Mine}, we saw that $\sat(n,K_r,t)=tn+O(1)$ (see page 1) and that all approximately extremal graphs are obtained in this way (see Theorem 2). More precisely, Theorem 2 is equivalent to the statement that every $K_r$-saturated graph $G$ on $n$ vertices with $\delta(G) \geq t$ and $e(G)=tn+O(1)$ is a blow-up of a $K_r$-saturated graph $H$ with $\delta(H) \geq t$ and $|H|=O(1)$. \newline

This leads us to define $\tsat(n,K_r,t)$ for integers $r \geq 3$, $t \geq r-2$ and large enough integers $n$ to be the minimum number of edges in a $K_r$-saturated graph $G$ on $n$ vertices with $\delta(G) \geq t$ that is minimal in the sense that it cannot be obtained by blowing up a smaller $K_r$-saturated graph $H$ with $\delta(H) \geq t$. Equivalently, $\tsat(n,K_r,t)$ is the minimum number of edges in a $K_r$-saturated graph $G$ on $n$ vertices with $\delta(G) \geq t$ in which every pair of twins is adjacent to a vertex of degree $t$. It is again not immediately clear whether such graphs even exist, but we will see later that they do for large enough $n$. \newline

Note that Theorem 2 in \cite{Mine} is equivalent to the statement that $\tsat(n,K_r,t)=tn+\omega(1)$. We show that for $t$ large enough relative to $r$, the $\omega(1)$ term tends to infinity like a sublinear power of $n$. While we do not determine the exact power, we show that is is roughly $1/(t-r)$.

\begin{theorem}\label{twin free min deg result}
We have the following estimates for $\tsat(n,K_r,t)$.
\begin{enumerate}
\item For all integers $r \geq 3$ and $t \geq r+3$, 
$$tn + \Omega\left(n^{1/(t-r+2)} \right) \leq \tsat(n,K_r,t) \leq tn+ O\left( n^{4/(t-r+2) } \right)  \ . $$
\item For all integers $t \geq 6$, 
$$tn + \Omega\left(n^{2/(t-3)} \right) \leq \tsat(n,K_3,t) \leq tn + O\left(n^{4/(t-1)}\right)  \ . $$
\item For all integers $r \geq 3$ and $r-2 \leq t \leq r+2$,
$$\left(r+2\right)n + o(n) \leq \tsat(n,K_r,t) \leq (r+3)n+o(n) \ .$$
\item For all integers $1 \leq t \leq 5$,
$$\left(5 + \frac{2}{3}\right)n + o(n) \leq \tsat(n,K_3,t) \leq 6n+o(n) \ .$$
\end{enumerate}
\end{theorem}

\newtheorem*{repth16}{Theorem \ref{twin free min deg result}}

Note that we again obtain stronger lower bounds in the special case $r=3$ and that for $t \leq r+2$ we obtain the same estimates for $\tsat(n,K_r,t)$ as for $\tsat(n,K_r)$ (see Theorem \ref{twin free result}). The author again conjectures that the upper bounds are tight. \newline

The rest of the paper is organised as follows. In section \ref{systems section} we first sketch the proof of the lower bound in part 2 of Theorem \ref{twin free result}. This then motivates the definition of $r$-systems and variants thereof, which are certain structures that arise naturally in the sketch proof. We then define certain extremal quantities in terms of these $r$-systems. Finally, we state several results providing estimates for these quantities and establishing relationships between them and between them and $\tsat(n,K_r,t)$. \newline

In section \ref{conical vertices second section} we first use conical vertices to obtain inequalities for $\tsat(n,K_r)$, $\tsat(n,K_r,t)$ and the quantities defined in section \ref{systems section}, which we will use in section \ref{twin free proof section} to prove Theorems \ref{twin free result} and \ref{twin free min deg result} and the results stated in section \ref{systems section}. We then state several conjectures regarding conical vertices and discuss the evidence for these conjectures. Next, in section \ref{twin free proof section}, we prove Theorems \ref{existence}, \ref{twin free result} and \ref{twin free min deg result} and the results stated in section \ref{systems section}. Finally, in section \ref{Open problems}, we state several open problems.

\section{Systems}\label{systems section}

In this section we first sketch the proof of the lower bound in part 2 of Theorem \ref{twin free result}. This then motivates the definition of $r$-systems and variants thereof, which are certain structures that arise naturally in the sketch proof. We then define certain extremal quantities in terms of these $r$-systems. Finally, we state several results providing estimates for these quantities and establishing relationships between them and between them and $\tsat(n,K_r,t)$.  \newline

We first sketch the proof of the lower bound in part 2 of Theorem \ref{twin free result}. Let $G$ be a twin-free $K_3$-saturated graph on $n$ vertices. We need to show that $e(G) \geq \left(5 + 2/3 \right)n + o(n)$. If $e(G)=\omega(n)$, there is nothing to prove, so we may assume $e(G)=O(n)$. Then by Theorem 4 in \cite{Mine}, $G$ has a vertex cover (a set of vertices intersecting every edge) $C$ of size $m=O(n/\log n)$. Let $H=G[C]$ be the subgraph induced by $C$ and $\mathcal{F}=\{\Gamma(v) : v \not \in C\}$ be the family of neighbourhoods of the vertices outside $C$. Note that the elements of $\mathcal{F}$ are sets of vertices in $H$ and that every vertex $v \not \in C$ gives a different element of $\mathcal{F}$, since $G$ is twin-free.  \newline

Since $G$ is $K_3$-free, $H$ is $K_3$-free and the sets in $\mathcal{F}$ are independent. Since $G$ is $K_3$-saturated, the sets in $\mathcal{F}$ are in fact maximally independent. Finally, since the vertices outside $C$ are non-adjacent and $G$ is $K_3$-saturated, $\mathcal{F}$ is an intersecting family. \newline

For each integer $t \geq 0$, let $\mathcal{F}_t=\{S \in \mathcal{F} : |S|=t \}$. We will show that for all integers $0 \leq t \leq 4$, $K_3$-free graphs $H$ with $|H|=m$ and intersecting families $\mathcal{F}$ of maximally independent sets of vertices in $H$ of size $t$, we have $|\mathcal{F}|=O(m)$. We will also show that for all $K_3$-free graphs $H$ with $|H|=m$ and intersecting families $\mathcal{F}$ of maximally independent sets of vertices in $H$ of size $5$, we have $e(H) \geq 2|\mathcal{F}|/3+o\left(|\mathcal{F}|\right)$. Hence $|\mathcal{F}_t|=O(n/\log n)$ for every integer $0 \leq t \leq 4$ and $e(H) \geq 2|\mathcal{F}_5|/3+o(n)$, from which it easily follows that $e(G) \geq \left(5 + 2/3 \right)n + o(n)$. \newline

We now define $r$-systems and variants thereof, which are generalisations to all integers $r \geq 3$ of the structures that arose naturally in our sketch proof. Suppose we have a graph $G$ and a vertex cover $C$ of $G$. Let $H=G[C]$ be the subgraph induced by $C$ and $\mathcal{F}=\{\Gamma(v) : v \not \in C\}$ be the family of neighbourhoods of the vertices outside $C$, as in the sketch proof. We no longer require $G$ to be twin-free, so we think of $\mathcal{F}$ as a family with `multiplicities' in the sense that for twins $v$ and $w$ outside $C$ we consider $\Gamma(v)$ and $\Gamma(w)$ to be distinct elements of $\mathcal{F}$ even though they are the same set. Note that the elements of $\mathcal{F}$ are again sets of vertices in $H$ and that $G$ is completely determined by $H$ and $\mathcal{F}$. Let us write $G(H,\mathcal{F})$ for this graph $G$. It is then easy to work out the conditions on $H$ and $\mathcal{F}$ for $G(H,\mathcal{F})$ to be $K_r$-saturated, which leads to the following definitions and observation. \newline

\begin{definition}\label{r system}
For an integer $r \geq 3$, an \emph{$r$-system} is a pair $(H,\mathcal{F})$, where $H$ is a graph and $\mathcal{F}$ is a family of sets of vertices in $H$ with multiplicities, with the following properties.

\begin{enumerate}
\item $H$ is $K_r$-free.
\item Every set $S \in \mathcal{F}$ is maximally $K_{r-1}$-free, meaning $H[S]$ is $K_{r-1}$-free but $H[S \cup \{v\}]$ contains $K_{r-1}$ for every vertex $v \not \in S$.
\item $H[S \cap T]$ contains $K_{r-2}$ for all $S \neq T \in \mathcal{F}$. 
\end{enumerate}

\end{definition}

\begin{definition}\label{maximal r system}
For an integer $r \geq 3$, a \emph{maximal $r$-system} is an $r$-system $(H,\mathcal{F})$ such that for every missing edge $e$ in $H$, either $H+e$ contains $K_r$ or $(H+e)[S]$ contains $K_{r-1}$ for some $S \in \mathcal{F}$.

\end{definition}

\begin{obs}\label{obs 1}
$G(H,\mathcal{F})$ is $K_r$-saturated if and only if $(H,\mathcal{F})$ is a maximal $r$-system.
\end{obs}

Note that if $(H,\mathcal{F})$ is an $r$-system, we can add edges to $H$ to obtain a graph $H'$ such that $(H',\mathcal{F})$ is a maximal $r$-system.

\vspace{0.5cm}

Recall that in the sketch proof, we partitioned $\mathcal{F}$ into parts $\mathcal{F}_t$ depending on the sizes of the sets. This motivates the following definition.

\begin{definition}\label{r t system}
For integers $r \geq 3$ and $t \geq r-2$, an \emph{$(r,t)$-system} or \emph{maximal $(r,t)$-system} is an $r$-system or maximal $r$-system $(H,\mathcal{F})$, respectively, such that $\mathcal{F}$ has no repeated elements and $|S|=t$ for all $S \in \mathcal{F}$.
\end{definition}

\vspace{0.5cm}

Recall also that in the sketch proof, we used an upper bound on $|\mathcal{F}|$ in terms of $|H|$ and a lower bound on $e(H)$ in terms of $|\mathcal{F}|$ for $(3,t)$-systems $(H,\mathcal{F})$. This motivates the following definitions.

\begin{definition}\label{set function definition}
For integers $r \geq 3$, $t \geq r-2$ and $m \geq 0$, let $s_{r,t}(m)$ be the maximum of $|\mathcal{F}|$ over all $(r,t)$-systems - or equivalently, over all maximal $(r,t)$-systems - $(H,\mathcal{F})$ with $|H|=m$.
\end{definition}

\begin{definition}\label{edge function definition}
For integers $r \geq 3$, $t \geq r$ and $s \geq 0$, let $e_{r,t}(s)$ or $e_{r,t}'(s)$ be the minimum of $e(H)$ over all $(r,t)$-systems or maximal $(r,t)$-systems $(H,\mathcal{F})$, respectively, with $|\mathcal{F}|=s$.
\end{definition}

It is easy to see that if $(H,\mathcal{F})$ is an $(r,t)$-system and $\mathcal{F}' \subseteq \mathcal{F}$ is a subfamily, then $(H,\mathcal{F}')$ is also an $(r,t)$-system. Hence $e_{r,t}(s)$ is an increasing function of $s$. As we shall see shortly, for all fixed $r$ and $t$ with $t \geq r$, $s_{r,t}(m)$ tends to infinity as $m$ tends to infinity. Hence, for all integers $r \geq 3$, $t \geq r$ and $s \geq 0$, there exist $(r,t)$-systems - and therefore also maximal $(r,t)$-systems - $(H,\mathcal{F})$ with $|\mathcal{F}|=s$, so $e_{r,t}(s)$ and $e_{r,t}'(s)$ are well-defined. For $t<r$, $s_{r,t}(m)$ is bounded, so we cannot define $e_{r,t}(s)$ and $e_{r,t}'(s)$.

\vspace{0.5cm}

Note that $e_{r,t}(s) \leq e_{r,t}'(s)$. The author conjectures that for fixed $r$ and $t$ and large $s$, $e_{r,t}(s)$ and $e_{r,t}'(s)$ have the same order of magnitude.

\begin{conj}\label{conj 4}
For all integers $r \geq 3$ and $t \geq r$, $e_{r,t}(s) \asymp e_{r,t}'(s)$.
\end{conj}

(We write $f \asymp g$ if $f=\Theta(g)$.) We will see some evidence for Conjecture \ref{conj 4} later. 

\vspace{0.5cm}

We obtain estimates for $s_{r,t}(m)$, which we will use in the proofs of Theorem \ref{twin free result} and parts 3 and 4 of Theorem \ref{twin free min deg result} and to prove estimates for $e_{r,t}(s)$ and $e_{r,t}'(s)$.

\begin{theorem}

We have the following estimates for $s_{r,t}(m)$.

\begin{enumerate}\label{set estimates}
\item For all integers $r \geq 3$, $s_{r,r-2}(m)=1$ for large enough $m$.
\item For all integers $r \geq 3$, $s_{r,r-1}(m)=2$ for large enough $m$.
\item For all integers $r \geq 3$, $s_{r,r}(m)=\lfloor (m-r+2)/2 \rfloor$ for large enough $m$.
\item For all integers $r \geq 3$, $s_{r,r+1}(m)=\Theta(m)$.
\item For all integers $r \geq 3$ and $t \geq r+2$,
$$m^{(t-r+2)/2} \ll s_{r,t}(m) \ll m^{t-r+2} \ . $$ 
\item For all integers $t \geq 5$,
$$m^{(t-1)/2} \ll s_{3,t}(m) \ll m^{t-3} \ . $$
\end{enumerate}

\end{theorem}

\newtheorem*{repth11}{Theorem \ref{set estimates}}

(We write $f \ll g$ if $f=O(g)$.) As we shall see later, the upper bound in part 5 is essentially trivial (see the proof of Theorem \ref{set estimates} in section \ref{twin free proof section}). The author conjectures that the lower bound is tight. Part 6 gives a better upper bound in the special case $r=3$.

\vspace{0.5cm}

We also obtain inequalities relating $s_{3,t}(m)$ and $s_{3,t+1}(m)$.

\begin{theorem}\label{set estimates 3 t to t+1 relation}
For all integers $t \geq 1$,
$$s_{3,t}(m) \ll s_{3,t+1}(m+1) \ll m \ s_{3,t}(m) \ . $$
\end{theorem}

\newtheorem*{repth14}{Theorem \ref{set estimates 3 t to t+1 relation}}

Part 6 of Theorem \ref{set estimates} follows from part 4 of Theorem \ref{set estimates} with $r=3$ by iterating the upper bound in Theorem \ref{set estimates 3 t to t+1 relation}. Note that the lower and upper bounds in part 6 of Theorem \ref{set estimates} match when $t=5$. This together with parts 1 through 4 of Theorem \ref{set estimates} with $r=3$ gives examples showing that both the lower and upper bound in Theorem \ref{set estimates 3 t to t+1 relation} can be tight. \newline

The following simple lemma will allow us to obtain lower bounds for $e_{r,t}(s)$ from upper bounds for $s_{r,t}(m)$ and upper bounds for $e_{r,t}'(s)$ from lower bounds for $s_{r,t}(m)$.

\begin{lemma}\label{set to edge lemma}
Let $r \geq 3$, $t \geq r-2$ and $m \geq 0$ be integers and $(H,\mathcal{F})$ be an $(r,t)$-system with $|H|=m$ and $|\mathcal{F}| \geq 1$. Then
$$m-t \leq e(H) \leq \binom{m}{2} \  .$$
\end{lemma}

\newtheorem*{repth13}{Lemma \ref{set to edge lemma}}

\vspace{1cm}

We obtain estimates for $e_{r,t}(s)$ and $e_{r,t}'(s)$ which we will use in the proofs of part 2 of Theorem \ref{twin free result} and parts 1, 2 and 4 of Theorem \ref{twin free min deg result}.

\begin{theorem}

We have the following estimates for $e_{r,t}(s)$ and $e_{r,t}'(s)$.

\begin{enumerate}\label{edge estimates}
\item For all integers $r \geq 3$, $e_{r,r}(s)=e_{r,r}'(s)=s^2+(2r-7)s+\binom{r-2}{2}$ for large enough $s$.
\item For all integers $r \geq 3$, $e_{r,r+1}(s), e_{r,r+1}'(s)=\Theta\left(s^{3/2}\right)$.
\item For all integers $r \geq 3$ and $t \geq r+2$,
$$s^{1/(t-r+2)} \ll e_{r,t}(s) \leq  e'_{r,t}(s)  \ll s^{4/(t-r+2)} \ . $$ 
\item For all integers $t \geq 5$,
$$s^{2/(t-3)} \ll e_{3,t}(s) \leq  e'_{3,t}(s) \ll s^{4/(t-1)} \ . $$
\item We have
$$\frac{2}{3} s+o(s) \leq e_{3,5}(s) \leq e_{3,5}'(s) \leq 2s+o(s) \ . $$
\end{enumerate}

\end{theorem}

\newtheorem*{repth17}{Theorem \ref{edge estimates}}

Part 3 is obtained from part 5 of Theorem \ref{set estimates} using Lemma \ref{set to edge lemma}. Part 4 gives a stronger lower bound in the special case $r=3$. While part 4 already gives $e_{3,5}(s)=\Theta(s)$, it turns out that the precise constant factor is closely related to $\tsat(n,K_3)$ and $\tsat(n,K_3,t)$ for $1 \leq t \leq 5$, which is the reason for including part 5.  The author conjectures that the upper bounds are tight. By parts 1, 2 and 5, Conjecture \ref{conj 4} is true when $t=r$, when $t=r+1$ and when $r=3$ and $t=5$.

\vspace{0.5cm}

We also show that $e_{3,t}(s)$ decreases with $t$.

\begin{theorem}\label{edge estimates 3 t to t+1 relation}
For all integers $t \geq 3$, $e_{3,t+1}(s) \leq e_{3,t}(s)$.
\end{theorem}

\newtheorem*{repth18}{Theorem \ref{edge estimates 3 t to t+1 relation}}

\vspace{0.5cm}

Finally, we show that for $t$ large enough relative to $r$, the $\omega(1)$ term in $\tsat(n,K_r,t)=tn+\omega(1)$ is essentially $e_{r,t}'(n)$.

\begin{theorem}\label{connection with edges}

For all integers $r \geq 3$ and $t \geq r+3$, 

$$tn + \Omega\left(e_{r,t}'\left[n+o(n)\right]\right) \leq \tsat(n,K_r,t) \leq tn + O\left(e_{r,t}'\left[n+o(n)\right]\right)  \ . $$

\end{theorem}

\newtheorem*{repth15}{Theorem \ref{connection with edges}}

Parts 1 and 2 of Theorem \ref{twin free min deg result} are obtained from Theorem \ref{connection with edges} and parts 3 and 4 of Theorem \ref{edge estimates}, respectively.

\section{Conical vertices}\label{conical vertices second section}

In this section we first use conical vertices to obtain inequalities for $\tsat(n,K_r)$, $\tsat(n,K_r,t)$, $s_{r,t}(m)$, $e_{r,t}(s)$ and $e'_{r,t}(s)$ which we will use in the proofs in section \ref{twin free proof section}. We then state several conjectures regarding conical vertices and discuss the evidence for these conjectures.  \newline

It is easy to see that for every graph $G$ and integers $r,s \geq 0$, the graph $G^s$ obtained by adding $s$ conical vertices to $G$ is $K_{r+s}$-saturated if and only if $G$ is $K_r$-saturated. It is also easy to check that for every graph $H$, family $\mathcal{F}$ of sets of vertices in $H$ and integers $r \geq 3$, $t \geq r-2$ and $s \geq 0$, $(H^s,\mathcal{F}^s)$ is an $(r+s,t+s)$-system or maximal $(r+s,t+s)$-system if and only if $(H,\mathcal{F})$ is an $(r,t)$-system or maximal $(r,t)$-system, respectively, where $\mathcal{F}^s$ is the family of sets of vertices in $H^s$ which are the union of a set in $\mathcal{F}$ and the $s$ new conical vertices. We thus immediately obtain the following four lemmas, which we will use later.

\begin{lemma}\label{twin-free-conical} For all integers $r \geq 3$, $s \geq 0$ and large enough integers $n$,
$$\tsat(n+s,K_{r+s}) \leq \tsat(n,K_r) + sn +  \binom{s}{2} \ . $$
\end{lemma}

\begin{lemma}\label{twin-free-conical-min-degree} For all integers $r \geq 3$, $t \geq r-2$, $s \geq 0$ and large enough integers $n$,

$$\tsat(n+s,K_{r+s},t+s) \leq \tsat(n,K_r,t) + sn +  \binom{s}{2} \ . $$
\end{lemma}

\begin{lemma}\label{lemma conical systems sets}
For all integers $r \geq 3$, $t \geq r-2$ and $s,m \geq 0$, 
$$s_{r+s,t+s}(m+s) \geq s_{r,t}(m) \ . $$
\end{lemma}

\begin{lemma}\label{lemma conical systems edges}
Let $r \geq 3$, $t \geq r$ and $s \geq 0$ be integers and $(H,\mathcal{F})$ be an $(r,t)$-system or maximal $(r,t)$-system with $|\mathcal{F}|=S$. Then 
$$e_{r+s,t+s}(S) \ \ \text{or} \ \ e_{r+s,t+s}'(S) \ \leq \ e(H)+s|H|+\binom{s}{2} \ ,$$
respectively.
\end{lemma}

By Lemma \ref{set to edge lemma}, $|H| \ll e(H)$ in Lemma \ref{lemma conical systems edges}, so we obtain the following corollary of Lemma \ref{lemma conical systems edges}.

\begin{col}\label{conical lemma edges asymp}
For all integers $r \geq 3$, $t \geq r$ and $s \geq 0$,
$$e_{r+s,t+s}(S) \ll e_{r,t}(S) \ \ \text{and} \ \ e_{r+s,t+s}'(S) \ll e_{r,t}'(S) \  . $$  
\end{col}

The upper bound in part 1 of Theorem \ref{twin free result} is obtained from the upper bound in part 2 using Lemma \ref{twin-free-conical}, the upper bound in part 3 of Theorem \ref{twin free min deg result} follows from the upper bound in part 4 using Lemma \ref{twin-free-conical-min-degree}, the lower bound in part 5 of Theorem \ref{set estimates} is obtained from the lower bound in part 6 using Lemma \ref{lemma conical systems sets}, and the upper bound in part 3 of Theorem \ref{edge estimates} follows from the upper bound in part 4 using the second part of Corollary \ref{conical lemma edges asymp}.

\vspace{0.5cm}

The author makes the following four conjectures, which are similar to Conjecture 1 in \cite{Mine}.

\begin{conj}\label{con vert twin free}
For all integers $r \geq 4$, every extremal graph for $\tsat(n,K_r)$ has a conical vertex for large enough $n$.
\end{conj}

\begin{conj}\label{con vert twin free min deg}
For all integers $r \geq 4$ and $t \geq r-2$, every extremal graph for $\tsat(n,K_r,t)$ has a conical vertex for large enough $n$.
\end{conj}

\begin{conj}\label{con vert set systems}
For all integers $r \geq 4$ and $t \geq r-2$, every underlying graph of an extremal system for $s_{r,t}(m)$ has a conical vertex for large enough $m$.
\end{conj}

\begin{conj}\label{con vert set systems edges}
For all integers $r \geq 4$ and $t \geq r$, every underlying graph of an extremal system for $e_{r,t}(s)$ or $e_{r,t}'(s)$ has a conical vertex for large enough $s$.
\end{conj}

Conjecture \ref{con vert twin free} would imply that we have equality in Lemma \ref{twin-free-conical}, or equivalently that

$$\tsat(n,K_r)=\tsat(n-r+3,K_3) + (r-3)n - \binom{r-2}{2} $$

for all integers $r \geq 3$ and large enough integers $n$. Similarly, Conjecture \ref{con vert twin free min deg} would imply that we have equality in Lemma \ref{twin-free-conical-min-degree}, or equivalently that

$$\tsat(n,K_r,t)=\tsat(n-r+3,K_3,t-r+3) + (r-3)n - \binom{r-2}{2} $$

for all integers $r \geq 3$, $t \geq r-2$ and large enough integers $n$. \newline

Suppose $(H,\mathcal{F})$ is an $(r,t)$-system, $v \in H$ is a conical vertex and $S \in \mathcal{F}$ is a set with $v \not \in S$. Then by condition 1 in Definition \ref{r system}, $H \setminus \{v\}$ is $K_{r-1}$-free (for a graph $G$ and a set of vertices $S \subseteq V(G)$, we let $G \setminus S=G[V(G) \setminus S]$). But by condition 2 in Definition \ref{r system}, $S$ is maximally $K_{r-1}$-free, so $S$ must be the entire vertex set of $H \setminus \{v\}$. Hence $|H|=t+1$ by Definition \ref{r t system}. It follows that if $(H,\mathcal{F})$ is an $(r,t)$-system, $v \in H$ is a conical vertex and $|H|$ is large enough then every set $S \in \mathcal{F}$ must contain $v$ and hence $(H,\mathcal{F})=((H \setminus \{v\})^1,\mathcal{G}^1)$, where $\mathcal{G}=\{S \setminus \{v\} : S \in \mathcal{F}\}$. \newline

So Conjecture \ref{con vert set systems} would imply that we have equality in Lemma \ref{lemma conical systems sets} for large enough $m$, or equivalently that $s_{r,t}(m)=s_{3,t-r+3}(m-r+3)$ for all integers $r \geq 3$, $t \geq r-2$ and large enough integers $m$. Similarly, Conjecture \ref{con vert set systems edges} would imply that we have equality in Corollary \ref{conical lemma edges asymp}, or equivalently that $e_{r,t}(s) \asymp e_{3,t-r+3}(s)$ and $e_{r,t}'(s) \asymp e_{3,t-r+3}'(s)$ for all integers $r \geq 3$ and $t \geq r$. By parts 1 and 2 of Theorem \ref{edge estimates}, this is true for $t=r$ and $t=r+1$, which is some evidence for Conjecture \ref{con vert set systems edges}. \newline

There is further evidence for Conjectures \ref{con vert twin free} through \ref{con vert set systems edges}. Hajnal showed that every $K_r$-saturated graph $G$ with $\delta(G) < 2(r-2)$ has a conical vertex (see Theorem 1 in \cite{Hajnal}). Since $\sat(n,K_r,t)=tn+O(1)$ (see page 1 in \cite{Mine}), it follows that $e(G) \geq 2(r-2)n+O(1)$ for every $K_r$-saturated graph $G$ on $n$ vertices with no conical vertex. Combined with the upper bound in part 1 of Theorem \ref{twin free result}, this shows that Conjecture \ref{con vert twin free} is true for $r \geq 8$. Similarly, combined with the upper bounds in parts 1 and 3 of Theorem \ref{twin free min deg result}, this shows that Conjecture \ref{con vert twin free min deg} is true when $r \geq 8$ and $t<2(r-2)$. \newline

Suppose $(H,\mathcal{F})$ is a maximal $(r,t)$-system with $t<2(r-2)$, $|H|>t$ and $|\mathcal{F}| \geq 1$. Then $G(H,\mathcal{F})$ is $K_r$-saturated by Observation \ref{obs 1} and has a vertex of degree $t<2(r-2)$, so $G(H,\mathcal{F})$ has a conical vertex $v$ by Hajnal's result. Since every vertex of $G(H,\mathcal{F})$ outside $H$ has degree $t<|H|$, $v$ must be in $H$. Since for all fixed $r$ and $t$, $s_{r,t}(m) \geq 1$ for large enough $m$ by Theorem \ref{set estimates}, it follows that Conjecture \ref{con vert set systems} is true for extremal \emph{maximal} systems when $t<2(r-2)$. Similarly, since $|H|$ must be large for every $(r,t)$-system $(H,\mathcal{F})$ with $|\mathcal{F}|$ large, it follows that Conjecture \ref{con vert set systems edges} is true for $e'_{r,t}(s)$ when $t<2(r-2)$. We will see later that Conjectures \ref{con vert set systems} and \ref{con vert set systems edges} are also true when $r=4=t$, the first case not covered by these results (see Lemma \ref{conical stability}). 

\section{Proofs}\label{twin free proof section}

In this section we prove Theorems \ref{existence} through \ref{connection with edges} and Lemma \ref{set to edge lemma}. We first prove Theorem \ref{existence}.

\begin{repthex}
Let $r \geq 3$ and $n \geq 0$ be integers. Then there exists a twin-free $K_r$-saturated graph on $n$ vertices unless $n=r$, $n=r+1$, $r=3$ and $n=6$, or $r=3$ and $n=7$.
\end{repthex}

\begin{proof}
If $n<r$, $K_n$, which is twin-free, is the unique $K_r$-saturated graph on $n$ vertices. It is easy to see that the unique $K_r$-saturated graph on $r$ vertices is the graph obtained by blowing up a vertex of $K_{r-1}$ by $2$. It is also easy to check that the only $K_r$-saturated graphs on $r+1$ vertices are the graph obtained by blowing up a vertex of $K_{r-1}$ by $3$ and the graph obtained by blowing up two vertices of $K_{r-1}$ by $2$. It is easy to see that the cycle $C_5$ is twin-free and $K_3$-saturated. Adding $r-3$ conical vertices gives a twin-free $K_r$-saturated graph on $r+2$ vertices for every integer $r \geq 3$. \newline

We now show that there is no twin-free $K_3$-saturated graph on $6$ or $7$ vertices. It is easy to see that the only $K_3$-saturated graph with a vertex of degree $0$ is $K_1$ and that all $K_3$-saturated graphs with a vertex of degree $1$ are stars $K_{1,m}$. It is also easy to check that all $K_3$-saturated graphs with a vertex of degree $2$ are either complete bipartite graphs $K_{2,m}$ or blow-ups of the cycle $C_5$. Moreover, it is easy to see that all $K_3$-saturated graphs on $n$ vertices with a vertex of degree $n-1$ are stars $K_{1,n-1}$. Furthermore, it is easy to check that all $K_3$-saturated graphs on $n$ vertices with a vertex of degree $n-2$ are complete bipartite graphs $K_{2,n-2}$ and that all $K_3$-saturated graphs on $n$ vertices with a vertex of degree $n-3$ are either complete bipartite graphs $K_{3,n-3}$ or blow-ups of the cycle $C_5$. It follows that there is no twin-free $K_3$-saturated graph on $6$ vertices and that a twin-free $K_3$-saturated graph on $7$ vertices would have to be $3$-regular, which is impossible. \newline

The vertices of $C_7$ have distinct closed neighbourhoods. Also, $C_7$ has no vertex cover of size $3$, but removing any edge gives the path $P_7$, which does have a vertex cover of size $3$. Equivalently, the complement $C_7^C$ of $C_7$ is twin-free and $K_4$-saturated. Adding $r-4$ conical vertices gives a twin-free $K_r$-saturated graph on $r+3$ vertices for every integer $r \geq 4$. Let $C_8$ have vertices $v_1, \cdots, v_8$, in cyclic order, and let $G$ be the graph obtained from $C_8$ by adding the chords $v_1 v_5$ and $v_2 v_6$. In other words, $G$ is the graph obtained from $C_8$ by adding two chords at distance $1$, both between opposite vertices. Then it is easy to check that the vertices of $G$ have distinct closed neighbourhoods and that $G$ has no vertex cover of size $4$, but removing any edge from $G$ gives a graph which does have a vertex cover of size $4$. Equivalently, $G^C$ is twin-free and $K_4$-saturated. Adding $r-4$ conical vertices gives a twin-free $K_r$-saturated graph on $r+4$ vertices for every integer $r \geq 4$. It is easy to check that the graph obtained from $C_8$ by adding all four chords between opposite vertices is twin-free and $K_3$-saturated. Adding $r-3$ conical vertices gives a twin-free $K_r$-saturated graph on $r+5$ vertices for every integer $r \geq 3$.  \newline

We now construct a twin-free $K_3$-saturated graph on $n$ vertices for every integer $n \geq 9$.  Adding $r-3$ conical vertices then gives a twin-free $K_r$-saturated graph on $r+k$ vertices for all integers $r \geq 3$ and $k \geq 6$. Let us say a set $S \subseteq \{0,1\}^k$ of binary sequences of length $k$ \emph{shatters} a pair $\{i,j\} \subseteq \{1,\cdots,k\}$ if for every binary sequence $y \in \{0,1\}^2$ there is a binary sequence $x \in S$ with $(x_i,x_j)=y$. For each integer $k \geq 2$, let $m(k)$ be the minimum size of a set $S \subseteq \{0,1\}^k$ that shatters every pair in $\{1,\cdots,k\}$. We will need a crude upper bound for $m(k)$. It is easy to check that $m(2)=4$ and $m(3)=4$. Let $E_k \subseteq \{0,1\}^k$ be the set of binary sequences of length $k$ with an even number of $1$ entries and $0_k \in \{0,1\}^k$ be the binary sequence with all entries equal to $0$. Then it is easy to check that $E_k \setminus \{0_k\}$ shatters all pairs in $\{1,\cdots,k\}$ for $k \geq 4$, so $m(k) \leq 2^{k-1}-1$ for $k \geq 4$. \newline

Let $n \geq 9$ be an integer. Let $k \geq 2$ be the smallest integer with $2^k+2k+1 \geq n$. We show that $n \geq m(k)+2k+1$ unless $n=10$. Indeed, if $k=2$, then $n=9$, so the inequality holds. If $k=3$, then $n \geq 10$, so the inequality holds unless $n=10$. Finally, if $k \geq 4$, then $n \geq 2^{k-1}+2k \geq m(k)+2k+1$. Hence, if $n \neq 10$, $m(k) \leq n-2k-1 \leq 2^k$, so there exists a set $S \subseteq \{0,1\}^k$ of size $n-2k-1$ that shatters all pairs in $\{1,\cdots,k\}$. Let $M=\{1,\cdots,k\} \times \{0,1\}$ and $G$ be the graph with vertex set $M \cup S \cup \{v\}$ and edges as follows. For each $i \in \{1,\cdots,k\}$, $(i,0)$ and $(i,1)$ are adjacent, every $x \in S$ is adjacent to $(i,x_i)$ for all $i \in \{1,\cdots,k\}$, $v$ is adjacent to all $x \in S$, and there are no other edges. Then it is easy to check that $G$ is a twin-free $K_3$-saturated graph on $n$ vertices. It remains to show that there is a twin-free $K_3$-saturated graph on $10$ vertices. It is easy to check that the Petersen graph - the graph whose vertices are the subsets of size $2$ of a set of size $5$ where two vertices are adjacent if and only if they are disjoint - has these properties.

\end{proof}

\vspace{1cm}

Our next aim is to prove Theorems \ref{set estimates 3 t to t+1 relation} and \ref{edge estimates 3 t to t+1 relation}. We first introduce three definitions, make an observation and prove a lemma.

\vspace{0.5cm}
 
Taking $r=3$ in Definition \ref{r system}, a $(3,t)$-system is a pair $(H,\mathcal{F})$, where $H$ is a graph and $\mathcal{F}$ is a family (with no repeated elements) of sets of vertices in $H$ of size $t$, with the following properties.

\begin{enumerate}
\item $H$ is $K_3$-free.
\item Every set $S \in \mathcal{F}$ is maximally independent.
\item $\mathcal{F}$ is an intersecting family.
\end{enumerate}

We now name similar structures that are not required to satisfy condition 3.

\begin{definition}\label{3 prime system definition}
For an integer $t \geq 1$, a \emph{$(3,t)'$-system} is a pair $(H,\mathcal{F})$, where $H$ is a graph and $\mathcal{F}$ is a family (with no repeated elements) of sets of vertices in $H$ of size $t$, with the following properties.

\begin{enumerate}

\item $H$ is $K_3$-free.
\item Every set $S \in \mathcal{F}$ is maximally independent.

\end{enumerate}

\end{definition}

We now define the analogues of $s_{3,t}(m)$ and $e_{3,t}(s)$.

\begin{definition}\label{prime function definition}
For integers $t \geq 1$ and $m \geq 0$, let $s_{3,t}'(m)$ be the maximum of $|\mathcal{F}|$ over all $(3,t)'$-systems $(H,\mathcal{F})$ with $|H|=m$.
\end{definition}

\begin{rem}
The problem of maximising the number of maximally independent sets of size $t$ in a graph on $m$ vertices becomes much easier if we do not require the graph to be $K_3$-free. Indeed, it is easy to see that the answer is $\Theta\left(m^t\right)$: the upper bound is trivial; for the lower bound, consider the disjoint union of $t$ cliques with sizes as equal as possible. In fact, Song and Yao obtained an exact result and showed that this graph is the unique extremal graph via a short inductive argument (see Theorem 1.1 in \cite{Song and Yao}).
\end{rem}

\begin{definition}\label{edge prime function definition}
For integers $t \geq 2$ and $s \geq 0$, let $e_{3,t}''(s)$ be the minimum of $e(H)$ over all $(3,t)'$-systems $(H,\mathcal{F})$ with $|\mathcal{F}|=s$.
\end{definition}

It is easy to see that if $(H,\mathcal{F})$ is a $(3,t)'$-system and $\mathcal{F}' \subseteq \mathcal{F}$ is a subfamily, then $(H,\mathcal{F}')$ is also a $(3,t)'$-system. Hence $e''_{3,t}(s)$ is an increasing function of $s$. We will see later that for all fixed $t \geq 2$, $s_{3,t}'(m)$ tends to infinity as $m$ tends to infinity. Hence, for all integers $t \geq 2$ and $s \geq 0$, there exist $(3,t)'$-systems $(H,\mathcal{F})$ with $|\mathcal{F}|=s$, so $e_{3,t}''(s)$ is well-defined. On the other hand, $s_{3,1}'(m)$ is bounded, so we cannot define $e_{3,1}''(s)$.

\vspace{1cm}

We will use the following observation in the proofs of Theorems \ref{set estimates 3 t to t+1 relation} and \ref{edge estimates 3 t to t+1 relation}.

\begin{obs}\label{induction observation}
Let $t \geq 1$ be an integer, $(H,\mathcal{F})$ be a $(3,t+1)'$-system, $v \in H$ be a vertex, $\Gamma'(v)=V(H) \setminus \left(\Gamma(v) \cup \{v\} \right)$ and $\mathcal{F}_v=\{S \setminus \{v\} : v \in S \in \mathcal{F} \}$. Then it is easy to check that $\left(H[\Gamma'(v)], \ \mathcal{F}_v\right)$ is a $(3,t)'$-system. \newline

Conversely, suppose $t \geq 1$ is an integer and $(H,\mathcal{F})$ is a $(3,t)'$-system. Let $H'$ be the graph obtained from $H$ by adding an isolated vertex $v$ and $\mathcal{F}'=\{S \cup \{v\} : S \in \mathcal{F}\}$. Then it is easy to check that $(H',\mathcal{F}')$ is a $(3,t+1)$-system.
\end{obs}

\vspace{1cm}

We now show that for fixed $t$, $s_{3,t}'(m)$ is approximately increasing. Throughout this section, for a graph $H$, family $\mathcal{F}$ of sets of vertices in $H$ and vertex $v \in H$, we denote by $s(v)$ the number of sets $S \in \mathcal{F}$ with $v \in S$.

\begin{lemma}\label{3 prime approx incr}
Let $M \geq m \geq t \geq 1$ be integers. Then $s_{3,t}'(M) \geq \left(1-t/m\right) s_{3,t}'(m)$.
\end{lemma}

\begin{proof}
Let $(H,\mathcal{F})$ be a $(3,t)'$-system with $|H|=m$ and $|\mathcal{F}|=s$. By double counting the number of pairs $(v,S)$, where $v \in H$ and $ S \in \mathcal{F}$, with $v \in S$, we obtain $\sum_{v \in H} s(v) = ts$. Hence, there is a vertex $v$ with $s(v) \leq ts/m$. Let $H'$ be the graph obtained from $H$ by blowing up $v$ by $M-m+1$ and $\mathcal{F}'=\{S \in \mathcal{F} : v \not \in S\}$. Then it is easy to check that $(H',\mathcal{F}')$ is a $(3,t)'$-system with $|H'|=M$ and $|\mathcal{F}'| \geq \left(1-t/m\right)s$, so the inequality follows. \newline

\end{proof}

\vspace{1cm}

We are now ready to prove Theorems \ref{set estimates 3 t to t+1 relation} and \ref{edge estimates 3 t to t+1 relation}. We first prove Theorem \ref{set estimates 3 t to t+1 relation}. In fact, we prove slightly more. We will use this in the proofs of the lower bounds in parts 3, 4 and 6 and the upper bound in part 4 of Theorem \ref{set estimates}.

\begin{repth14}
For all integers $t \geq 1$,

$$s_{3,t}(m) \leq s_{3,t}'(m) \leq s_{3,t+1}(m+1) \ll s_{3,t}'(m) \ll m \ s_{3,t}(m) \ . $$
\end{repth14}

\begin{proof}

It is clear from Definitions \ref{set function definition} and \ref{prime function definition} that $s_{3,t}(m) \leq s_{3,t}'(m)$. By the second part of Observation \ref{induction observation}, $s_{3,t}'(m) \leq s_{3,t+1}(m+1)$. \newline

We now show that $s_{3,t+1}(m+1) \ll s_{3,t}'(m)$. Let $(H,\mathcal{F})$ be a $(3,t+1)$-system with $|H|=m+1$. We need to show that $|\mathcal{F}| \ll s_{3,t}'(m)$. If $\mathcal{F}=\emptyset$, there is nothing to prove, so suppose otherwise. Let $S \in \mathcal{F}$. Then by condition 3 in Definition \ref{r system}, every set in $\mathcal{F}$ contains a vertex in $S$, so by a union bound, $|\mathcal{F}| \leq \sum_{v \in S} s(v)$. So it suffices to show that $s(v) \ll s_{3,t}'(m)$ for every vertex $v \in H$. \newline

Let $v \in H$. Define $\Gamma'(v)$ and $\mathcal{F}_v$ as in the first part of Observation \ref{induction observation}. Then by the first part of Observation \ref{induction observation}, $s(v)=|\mathcal{F}_v| \leq s_{3,t}'\left(|\Gamma'(v)|\right)$. Note that $0 \leq |\Gamma'(v)| \leq m$. Using Lemma \ref{3 prime approx incr} and the fact that $s_{3,t}'(l)=0$ for $l<t$, $s_{3,t}'(l)=1$ for $l=t$ and $s_{3,t}'(l) \geq 1$ for $l>t$, it is easy to show that $s_{3,t}'(l) \ll s_{3,t}'(m)$ for all integers $0 \leq l \leq m$, so $s(v) \ll s_{3,t}'(m)$. \newline

Finally, we show that $s_{3,t}'(m) \leq m \ s_{3,t}(m) / t$. Let $(H,\mathcal{F})$ be a $(3,t)'$-system with $|H|=m$. We need to show that $|\mathcal{F}| \leq m \ s_{3,t}(m)/t$. For each vertex $v \in H$, let $\mathcal{F}_v=\{S \in \mathcal{F} : v \in S\}$. Then it is easy to see that $(H,\mathcal{F}_v)$ is a $(3,t)$-system, so $s(v)=|\mathcal{F}_v| \leq s_{3,t}(m)$. Summing over all $v$ completes the proof.

\end{proof}

\vspace{0.5cm}

We now prove Theorem \ref{edge estimates 3 t to t+1 relation}. We again prove more. We will use this in the proofs of the lower bounds in parts 2 and 4 of Theorem \ref{edge estimates}. The proof will be similar to that of Theorem \ref{set estimates 3 t to t+1 relation}.

\begin{repth18}\leavevmode

\begin{enumerate}
\item For all integers $t \geq 2$,
$$ e_{3,t}(s) \geq e_{3,t}''(s) \geq e_{3,t+1}(s) \geq e_{3,t}''\left[\Omega(s)\right] \ . $$
(The first inequality only applies for $t \geq 3$, since $e_{3,t}(s)$ is only defined for $t \geq 3$.)
\item Let $t \geq 3$ be an integer and $(H,\mathcal{F})$ be a $(3,t)'$-system with $|H|=m \geq 1$, $e(H)=e$ and $|\mathcal{F}|=s$. Then $e \geq e_{3,t}(k)$ for some $k \geq ms/(2e+m)$.
\end{enumerate}
\end{repth18}

\begin{proof}

We first prove part 1. It is clear from Definitions \ref{edge function definition} and \ref{edge prime function definition} that $e_{3,t}(s) \geq e_{3,t}''(s)$. By the second part of Observation \ref{induction observation}, $e_{3,t}''(s) \geq e_{3,t+1}(s)$. \newline

We now show that $e_{3,t+1}(s) \geq e_{3,t}''\left[\Omega(s)\right]$. Let $(H,\mathcal{F})$ be a $(3,t+1)$-system with $|\mathcal{F}|=s$. We need to show that $e(H) \geq e_{3,t}''\left[\Omega(s)\right]$. Let $S \in \mathcal{F}$. Then, as in the proof of Theorem \ref{set estimates 3 t to t+1 relation}, by condition 3 in Definition \ref{r system}, every set in $\mathcal{F}$ contains a vertex in $S$, so there is a vertex $v \in S$ with $s(v) \gg s$. Define $\Gamma'(v)$ and $\mathcal{F}_v$ as in the first part of Observation \ref{induction observation}. Then by the first part of Observation \ref{induction observation}, $\left(H[\Gamma'(v)], \ \mathcal{F}_v\right)$ is a $(3,t)'$-system, so $e(H) \geq e\left(H[\Gamma'(v)]\right) \geq  e_{3,t}''\left[s(v)\right]=e_{3,t}''\left[\Omega(s)\right]$. \newline

We now prove part 2. Since $\sum_{v \in H} d(v) = 2e$, there is a vertex $v \in H$ with degree $d(v) \leq 2e/m$. Then by condition 2 in Definition \ref{3 prime system definition}, every set in $\mathcal{F}$ contains a vertex in $\Gamma(v) \cup \{v\}$, so there is a vertex $w \in \Gamma(v) \cup \{v\}$ with $s(w) \geq ms/(2e+m)$. Let $\mathcal{F}_w=\{S \in \mathcal{F} : w \in S\}$. Then, as in the proof of Theorem \ref{set estimates 3 t to t+1 relation}, it is easy to see that $(H,\mathcal{F}_w)$ is a $(3,t)$-system, so $e \geq e_{3,t}[s(w)]$.

\end{proof}

\vspace{1cm}

We now construct $(3,t)'$-systems which we will use in the proofs of the lower bounds in parts 3 and 6 and the upper bound in part 3 of Theorem \ref{set estimates}, the lower bound in part 1 and the upper bounds in parts 1 and 5 of Theorem \ref{edge estimates}, and the upper bound in part 2 of Theorem \ref{twin free result}.

\begin{constr}\label{system construction}

For each integer $t \geq 2$ and large enough integer $l$, we construct a $(3,t)'$-system $\left(H_{t,l},\mathcal{F}_{t,l}\right)$.

\begin{enumerate}

\item Let $H_{2,l}$ be the graph obtained from $K_{\lfloor l/2 \rfloor,\lceil l/2 \rceil}$ by removing a matching of size $\lfloor l/2 \rfloor$ and $\mathcal{F}_{2,l}$ be the family of pairs of matched vertices.

\item Let $H_{3,l}$ be the graph obtained from $K_{l,l}$ by removing a copy of $C_{2l}$ and $\mathcal{F}_{3,l}$ be the family of triples of consecutive vertices in the copy of $C_{2l}$. 

\item Let $S$ be a set of size $l$ and $X$ and $Y$ be copies of $S^2$. Let $H_{4,l}$ be the bipartite graph with parts $X$ and $Y$ where $(a,b) \in X$ is adjacent to $(c,d) \in Y$ if and only if either $a \neq c$ and $b \neq d$ or $a=c$ and $b=d$. Now define the family $\mathcal{F}_{4,l}$ as follows. For $a,b,c,d \in S$ with $a \neq c$ and $b \neq d$, the set consisting of the two vertices $(a,b), (c,d) \in X$ and the two vertices $(a,d), (c,b) \in Y$ is in $\mathcal{F}_{4,l}$ and there are no other sets in $\mathcal{F}_{4,l}$.

\item For $t \geq 5$, let $V_t=V(C_t)$ and $E_t=E(C_t)$. For each edge $e \in E_t$, let $S_e$ be a set of size $l$. Define the graph $H_{t,l}$ as follows. The vertex set of $H_{t,l}$ is $\bigcup_{v \in V_t} \prod_{e \ni v} S_e$, for every edge $vw \in E_t$, a vertex in $\prod_{e \ni v} S_e$ is adjacent to a vertex in $\prod_{e \ni w} S_e$ if and only if their $S_{vw}$ coordinates differ, and there are no other edges. Now define the family $\mathcal{F}_{t,l}$ as follows. For each vector $(s_e)_{e \in E_t} \in \prod_{e \in E_t} S_e$, the set $\{(s_e)_{e \ni v} : v \in V_t \}$ is in $\mathcal{F}_{t,l}$ and there are no other sets in $\mathcal{F}_{t,l}$.

\end{enumerate}

It is easy to check that $H_{t,l}$ and $\mathcal{F}_{t,l}$ have the following properties, which we will use later.

\begin{enumerate}

\item $\left(H_{t,l},\mathcal{F}_{t,l}\right)$ is a $(3,t)'$-system.
\item $|H_{2,l}|=l$, $|H_{3,l}|=2l$, $|H_{4,l}|=2l^2$ and $|H_{t,l}|=tl^2$ for $t \geq 5$.
\item $e(H_{2,l})=\lceil (l^2-2l)/4 \rceil$, $e\left(H_{4,l}\right)=l^4-2l^3+2l^2$ and $H_{t,l}$ is $2l(l-1)$-regular for $t \geq 5$.
\item $|\mathcal{F}_{2,l}|=\lfloor l/2 \rfloor$, $|\mathcal{F}_{3,l}|=2l$, $|\mathcal{F}_{4,l}|=l^2(l-1)^2/2$ and $|\mathcal{F}_{t,l}|=l^t$ for $t \geq 5$.
\item Let $(H_{t,l}',\mathcal{F}_{t,l}')$ be the $(3,t+1)$-system obtained from $\left(H_{t,l},\mathcal{F}_{t,l}\right)$ using the second part of Observation \ref{induction observation}. Then $(H_{t,l}',\mathcal{F}_{t,l}')$ is in fact a maximal $(3,t+1)$-system, unless $t=2$ and $l$ is odd. Moreover, there are only $O\left(l^3\right)$ missing edges $e$ in $H_{4,l}'$ such that $H_{4,l}'+e$ is $K_3$-free. Furthermore, $H_{t,l}'$ is twin-free for $t \geq 5$. 
\end{enumerate}

\end{constr}

\vspace{1cm}

Our next aim is to prove Theorems \ref{set estimates} and \ref{edge estimates}. We first prove Lemma \ref{set to edge lemma} and five other lemmas.

\begin{repth13}
Let $r \geq 3$, $t \geq r-2$ and $m \geq 0$ be integers and $(H,\mathcal{F})$ be an $(r,t)$-system with $|H|=m$ and $|\mathcal{F}| \geq 1$. Then
$$m-t \leq e(H) \leq \binom{m}{2} \  .$$
\end{repth13}

\begin{proof}
Let $S \in \mathcal{F}$. Then by condition 2 in Definition \ref{r system}, every vertex outside $S$ is fully connected to a clique of size $r-2$ in $S$. In particular, every vertex outside $S$ is adjacent to a vertex in $S$, so the lower bound follows. The upper bound is trivial.

\end{proof}

\vspace{1cm}

We will use the following structural result in the proof of the lower bound in part 5 of Theorem \ref{edge estimates}.

\begin{lemma}\label{principal intersection}
Let $t \geq 2$ be an integer and $(H,\mathcal{F})$ be a $(3,t+1)$-system with $s(v) \geq 1$ for all $v \in H$ and $|\mathcal{F}|=s \geq 1$. Then either $(H,\mathcal{F})$ can be obtained from a $(3,t)'$-system $(H',\mathcal{F}')$ using the second part of Observation \ref{induction observation} or there exists a  $(3,t-1)'$-system $(H',\mathcal{F}')$ with $H' \subseteq H$ and $|\mathcal{F}'|=\Omega(s)$.
\end{lemma}

\begin{proof}
Suppose first that $\bigcap_{S \in \mathcal{F}} S \neq \emptyset$. Pick a vertex $v \in \bigcap_{S \in \mathcal{F}} S$. Since $s(w) \geq 1$ for all $w \in H$, $v$ must be an isolated vertex by condition 2 in Definition \ref{r system}. Hence, by the first part of Observation \ref{induction observation}, $(H',\mathcal{F}')$ is a $(3,t)'$-system, where $H'=H \setminus \{v\}$ and $\mathcal{F}'=\{S \setminus \{v\} : S \in \mathcal{F}\}$. Then $(H,\mathcal{F})$ can be obtained from $(H',\mathcal{F}')$ using the second part of Observation \ref{induction observation}. \newline

Now suppose $\bigcap_{S \in \mathcal{F}} S = \emptyset$. We show that in fact $\bigcap_{S \in \mathcal{F}'} S = \emptyset$ for some subfamily $\mathcal{F}' \subseteq \mathcal{F}$ with $|\mathcal{F}'|=O(1)$. Pick a set $S \in \mathcal{F}$. Then, for each vertex $v \in S$, since $\bigcap_{T \in \mathcal{F}} T = \emptyset$, there is a set $S_v \in \mathcal{F}$ with $v \not \in S_v$. Then $\mathcal{F}'=\{S_v : v \in S\} \cup \{S\}$ is the desired subfamily. \newline
 
Let $S \in \mathcal{F}$. Then by condition 3 in Definition \ref{r system}, for every set $T \in \mathcal{F}'$ there is a vertex $v_T \in S \cap T$. Since $\bigcap_{T \in \mathcal{F}'} T = \emptyset$, the $v_T$ cannot all be the same vertex. Hence, $P \subseteq S$ for some pair $P \subseteq  \bigcup_{T \in \mathcal{F}'} T$. Since there are only $O(1)$ such pairs $P$, it follows that for some pair $P$ there are $\Omega(s)$ sets $S \in \mathcal{F}$ with $P \subseteq S$. Then, applying the first part of Observation \ref{induction observation} twice, we obtain that $(H',\mathcal{F}')$ is a $(3,t-1)'$-system, where $H'=H \setminus \bigcup_{v \in P} \left( \Gamma(v) \cup \{v\} \right)$  and $\mathcal{F}'=\{S \setminus P : P \subseteq S \in \mathcal{F} \}$.

\end{proof}

\vspace{1cm}

Let $(H_{2,l},\mathcal{F}_{2,l})$ be as in Construction \ref{system construction}. Then by property 1 in Construction \ref{system construction} and the second part of Observation \ref{induction observation}, we obtain a $(3,3)$-system $(H_{2,l}',\mathcal{F}_{2,l}')$. The following stability result asserts that this is essentially the only $(3,3)$-system with many sets in $\mathcal{F}$. We will use this in the proofs of the upper bound in part 3 of Theorem \ref{set estimates} and the lower bound in part 1 of Theorem \ref{edge estimates}.

\begin{lemma}\label{3 3 stability}

Let $(H,\mathcal{F})$ be a $(3,3)$-system with $|\mathcal{F}|=\omega(1)$ and $s(v) \geq 1$ for all $v \in H$. Then $(H,\mathcal{F})=(H_{2,l}',\mathcal{F}_{2,l}')$ for some even $l$.

\end{lemma}

\begin{proof}

By Lemma \ref{principal intersection}, either $(H,\mathcal{F})$ can be obtained from a $(3,2)'$-system $(H',\mathcal{F}')$ using the second part of Observation \ref{induction observation} or there exists a  $(3,1)'$-system $(H',\mathcal{F}')$ with $|\mathcal{F}'|=\omega(1)$. But the latter is impossible, since a $(3,1)'$-system is a family of conical vertices in a $K_3$-free graph, so $s_{3,1}'(m) \leq 2$ for all $m$. So it suffices to show that for every $(3,2)'$-system $(H',\mathcal{F}')$ with $|\mathcal{F}'|=\omega(1)$ and $s(v) \geq 1$ for all $v \in H'$, $(H',\mathcal{F}')=(H_{2,l},\mathcal{F}_{2,l})$ for some even $l$. \newline

Let $(H',\mathcal{F}')$ be a $(3,2)'$-system with $|\mathcal{F}'|=\omega(1)$ and $s(v) \geq 1$ for all $v \in H'$. Suppose for the sake of contradiction that there are two distinct intersecting sets in $\mathcal{F}'$, say $\{u,v\}$ and $\{v,w\}$. Then by condition 2 in Definition \ref{3 prime system definition}, $u$ and $w$ must be adjacent. Hence, by condition 1 in Definition \ref{3 prime system definition}, for every vertex $x \in V(H') \setminus \{u,v,w\}$, $x$ cannot be adjacent to both $u$ and $w$, so by condition 2 in Definition \ref{3 prime system definition}, $x$ must be adjacent to $v$. Hence, by condition 1 in Definition \ref{3 prime system definition}, $V(H') \setminus \{u,v,w\}$ is an independent set. \newline

So, by condition 2 in Definition \ref{3 prime system definition}, there is at most one set $S \in\mathcal{F}'$ with $S \subseteq V(H') \setminus \{u,v,w\}$. For every vertex $x \in H'$, if we define $\Gamma'(x)$ and $\mathcal{F}'_x$ as in the first part of Observation \ref{induction observation}, then applying the first part of Observation \ref{induction observation}, we obtain that $\left(H[\Gamma'(x)],\mathcal{F}'_x\right)$ is a $(3,1)'$-system, so there are at most $2$ sets $S \in \mathcal{F}'$ with $x \in S$, since $s_{3,1}'(m) \leq 2$ for all $m$. Hence, there are $O(1)$ sets $S \in \mathcal{F}'$ with $S \cap \{u,v,w\} \neq \emptyset$. It follows that $|\mathcal{F}'|=O(1)$, which is a contradiction. \newline

Hence the sets in $\mathcal{F}'$ partition $V(H')$, since $s(v) \geq 1$ for all $v \in H'$. By condition 2 in Definition \ref{3 prime system definition}, there are no edges within the parts. Let $R \neq S \in \mathcal{F}'$ and $u \in R$. On the one hand, by condition 2 in Definition \ref{3 prime system definition}, $u$ has at least one neighbour in $S$. On the other hand, we can pick a set $T \in \mathcal{F}' \setminus \{R,S\}$. Then $u$ has a neighbour $v \in T$, which has a neighbour $w \in  S$. Then by condition 1 in Definition \ref{3 prime system definition}, $u$ is not adjacent to $w$, so $u$ has at most one neighbour in $S$. Hence the induced graph between any two parts is a matching. \newline

Now let $R,S,T \in \mathcal{F}'$ be distinct and $u \in R$ not be matched to $v \in S$ and $w \in T$. Let $x$ and $y$ be the other vertices in $S$ and $T$, respectively. Then $u \in R$ is matched to $x \in S$ and $y \in T$, so by condition 1 in Definition \ref{3 prime system definition}, $x \in S$ and $y \in T$ are not matched. Hence $x \in S$ is matched to $w \in T$ and $y \in T$ is matched to $v \in S$, so $v \in S$ and $w \in T$ are not matched. It follows that not being matched is transitive, which implies that $(H',\mathcal{F}')=(H_{2,l},\mathcal{F}_{2,l})$, where $l=|H'|$ is even.

\end{proof}

\vspace{1cm}

The following structural result confirms Conjectures \ref{con vert set systems} and \ref{con vert set systems edges} when $t=r$. We will use this in the proofs of the upper bound in part 3 of Theorem \ref{set estimates} and the lower bound in part 1 of Theorem \ref{edge estimates}.

\begin{lemma}\label{conical stability}

Let $r \geq 3$ be an integer and $(H,\mathcal{F})$ be an $(r,r)$-system with $|\mathcal{F}|=\omega(1)$. Then $H$ has $r-3$ conical vertices.

\end{lemma}

\begin{proof}

Let $S \in \mathcal{F}$. Then by condition 3 in Definition \ref{r system}, for every set $T \in \mathcal{F} \setminus \{S\}$, there is a clique $K \subseteq S \cap T$ of size $r-2$. Since there are only $O(1)$ cliques $K \subseteq S$ of size $r-2$, it follows that there is a clique $K$ of size $r-2$ such that there are $\omega(1)$ sets $T \in \mathcal{F}$ with $K \subseteq T$. \newline

For each subset $R \subseteq K$, let $V_R=\{v \in V(H) \setminus K : \Gamma(v) \cap K =R\}$. Note that the $V_R$ partition  $V(H) \setminus K$. By Definition \ref{r t system}, every set $K \subseteq T \in \mathcal{F}$ is of the form $T=K \cup \{v,w\}$ for some vertices $v,w \in V(H) \setminus K$. Since there are only $O(1)$ subsets $R \subseteq K$, it follows that for some subsets $A, B \subseteq K$, there are $\omega(1)$ sets $T \in \mathcal{F}$ of the form $T=K \cup \{v,w\}$ for some vertices $v \in V_A$ and $w \in V_B$. \newline

Let us call a vertex in a graph \emph{almost conical} if it is not adjacent to at most two other vertices. It is easy to see that a graph with many almost conical vertices has a large clique, so by condition 1 in Definition \ref{r system}, there are only $O(1)$ almost conical vertices in $H[V_A]$ and $H[V_B]$. Hence, without loss of generality, there is a set $T \in \mathcal{F}$ of the form $T=K \cup \{v,w\}$ for some not almost conical vertex $v \in H[V_A]$ and vertex $w \in V_B$. We show that $A=B$, $|A|=r-3$ and $v$ and $w$ are not adjacent. \newline

By condition 2 in Definition \ref{r system}, $H[T]$ is $K_{r-1}$-free, so $A,B \neq K$. Since $v$ is not almost conical in $H[V_A]$, there is a vertex $u \in V_A \setminus \left(\Gamma(v) \cup \{v,w\} \right)$. Then by condition 2 in Definition \ref{r system}, $H[T \cup \{u\}]$ contains $K_{r-1}$. Since $u$ is not adjacent to $v$ and $A,B \neq K$, this forces $A=B$ and $|A|=r-3$. Since $H[T]$ is $K_{r-1}$-free, this implies $v$ and $w$ are not adjacent. \newline

For every vertex $u \not \in T$, by condition 2 in Definition \ref{r system}, $H[T \cup \{u\}]$ contains $K_{r-1}$, which implies $A \subseteq \Gamma(u)$, since $|A|=r-3$ and $v$ and $w$ are not adjacent. Since $K$ is a clique and $A \subseteq \Gamma(v) \cap \Gamma(w)$, it follows that all the vertices in $A$ are conical.

\end{proof}

\vspace{1cm}

We will use the following lemma in the proofs of the upper bound in part 4 of Theorem \ref{set estimates} and the lower bound in part 2 of Theorem \ref{edge estimates}. The proof will be similar to that of Lemma \ref{conical stability}, but more involved.

\begin{lemma}\label{r r+1 stability}

Let $r \geq 3$ be an integer and $(H,\mathcal{F})$ be an $(r,r+1)$-system with $|\mathcal{F}|=s$. Then there exists a $(3,2)'$-system or $(3,3)'$-system $(H',\mathcal{F}')$ with $H' \subseteq H$ and $|\mathcal{F}'|=\Omega(s)$.

\end{lemma}

\begin{proof}

Let $S \in \mathcal{F}$. Then by condition 3 in Definition \ref{r system}, for every set $T \in \mathcal{F} \setminus \{S\}$, there is a clique $K \subseteq S \cap T$ of size $r-2$. Since there are only $O(1)$ cliques $K \subseteq S$ of size $r-2$, it follows that there is a clique $K$ of size $r-2$ such that there are $\Omega(s)$ sets $T \in \mathcal{F}$ with $K \subseteq T$, as in the proof of Lemma \ref{conical stability}. \newline

For each subset $R \subseteq K$, let $V_R=\{v \in V(H) \setminus K : \Gamma(v) \cap K =R\}$. Note that the $V_R$ partition  $V(H) \setminus K$. By Definition \ref{r t system}, every set $K \subseteq T \in \mathcal{F}$ is of the form $T=K \cup \{u,v,w\}$ for some vertices $u,v,w \in V(H) \setminus K$. Since there are only $O(1)$ subsets $R \subseteq K$, it follows that for some subsets $A, B,C \subseteq K$, there are $\Omega(s)$ sets $T \in \mathcal{F}$ of the form $T=K \cup \{u,v,w\}$ for some vertices $u \in V_A$, $v \in V_B$ and $w \in V_C$, as before. \newline 

As in the proof of Lemma \ref{conical stability}, let us call a vertex in a graph almost conical if it is not adjacent to at most two other vertices. Then either there are $\Omega(s)$ sets $T \in \mathcal{F}$ of the form $T=K \cup \{u,v,w\}$ with $u$, $v$ and $w$ not almost conical in $H[V_A]$, $H[V_B]$ and $H[V_C]$, respectively, or there are $\Omega(s)$ sets $T \in \mathcal{F}$ of the form $T=K \cup \{u,v,w\}$ with $u$, $v$ or $w$ almost conical in $H[V_A]$, $H[V_B]$ or $H[V_C]$, respectively. \newline

In the first case, we show that $A=B=C$, $|A|=r-3$ and $\{u,v,w\}$ is always a maximally independent set in $H[V_A]$. Then by condition 1 in Definition \ref{r system}, $H[V_A]$ is $K_3$-free, so $(H[V_A],\mathcal{F}')$ is the desired $(3,3)'$-system, where $\mathcal{F}'$ is the family of triples $\{u,v,w\} \subseteq V_A$ with $K \cup \{u,v,w\} \in \mathcal{F}$ and $u$, $v$ and $w$ not almost conical in $H[V_A]$. Let $T \in \mathcal{F}$ be of the form $T=K \cup \{u,v,w\}$ for some not almost conical vertices $u \in H[V_A]$, $v \in H[V_B]$ and $w \in H[V_C]$. We need to show that $A=B=C$, $|A|=r-3$ and $\{u,v,w\}$ is a maximally independent set in $H[V_A]$. \newline

We first show that every vertex in $\{u,v,w\}$ is not adjacent to some other vertex in $\{u,v,w\}$. Let us prove this for $u$, with the other cases following by symmetry. Since $u$ is not almost conical in $H[V_A]$, there is a vertex $x \in V_A \setminus \left(\Gamma(u) \cup \{u,v,w\} \right)$. Then by condition 2 in Definition \ref{r system}, $H[T]$ is $K_{r-1}$-free but $H[T \cup \{x\}]$ contains $K_{r-1}$. Since $x$ is not adjacent to $u$, this implies $u$ is not adjacent to $v$ or not adjacent to $w$, for otherwise $\Gamma(x) \cap T \subseteq \Gamma(u)$.  \newline

Hence every vertex in $\{u,v,w\}$ is not adjacent to some other vertex in $\{u,v,w\}$, so without loss of generality, $u$ and $v$ are not adjacent and $v$ and $w$ are not adjacent. Suppose for the sake of contradiction that $u$ is adjacent to $w$. By condition 2 in Definition \ref{r system}, $H[T]$ is $K_{r-1}$-free, so $A,B,C \neq K$ and $|A \cap C| \leq r-4$. Again, since $u$ is not almost conical in $H[V_A]$, there is a vertex $x \in V_A \setminus \left(\Gamma(u) \cup \{u,v,w\} \right)$. Then by condition 2 in Definition \ref{r system}, $H[T]$ is $K_{r-1}$-free but $H[T \cup \{x\}]$ contains $K_{r-1}$. Since $x$ is not adjacent to $u$, $v$ is not adjacent to $w$, $A,B \neq K$ and $|A \cap C| \leq r-4$, this implies $A=B$ and $|A|=r-3$. By symmetry, we also have $C=B$ and $|C|=r-3$, so $A=C$ and $|A|=r-3$, which contradicts $|A \cap C| \leq r-4$. \newline

Hence $\{u,v,w\}$ is an independent set. We now show that $A=B=C$ and $|A|=r-3$. By condition 2 in Definition \ref{r system}, $H[T]$ is $K_{r-1}$-free, so $A,B,C \neq K$. Once again, since $u$ is not almost conical in $H[V_A]$, there is a vertex $x \in V_A \setminus \left(\Gamma(u) \cup \{u,v,w\} \right)$, so by condition 2 in Definition \ref{r system}, $H[T \cup \{x\}]$ contains $K_{r-1}$. Since $x$ is not adjacent to $u$, $v$ is not adjacent to $w$ and $A,B,C \neq K$, this implies that $|A|=r-3$ and either $A=B$ or $A=C$. Without loss of generality, suppose $A=B$. Then by symmetry, we also have $|C|=r-3$ and either $C=A$ or $C=B$, so $A=B=C$ and $|A|=r-3$. \newline

Finally, we show that $\{u,v,w\}$ is maximally independent set in $H[V_A]$. Let $x \in V_A \setminus \{u,v,w\}$. Then by condition 2 in Definition \ref{r system}, $H[T \cup \{x\}]$ contains $K_{r-1}$. Since $A \neq K$ and $\{u,v,w\}$ is an independent set, this implies that $x$ is adjacent to $u$, $v$ or $w$. \newline

Now consider the second case where there are $\Omega(s)$ sets $T \in \mathcal{F}$ of the form $T=K \cup \{u,v,w\}$ with $u$, $v$ or $w$ almost conical in $H[V_A]$, $H[V_B]$ or $H[V_C]$, respectively. As in the proof of Lemma \ref{conical stability}, by condition 1 in Definition \ref{r system}, there are only $O(1)$ almost conical vertices in $H[V_A]$, $H[V_B]$ and $H[V_C]$. Hence there is a vertex $u$ such that there are $\Omega(s)$ sets $T \in \mathcal{F}$ of the form $T=K \cup \{u,v,w\}$ for some vertices $v$ and $w$. \newline

Let $S=K \cup \{u\}$ and for each subset $R \subseteq S$, let $V_R=\{v \in V(H) \setminus S : \Gamma(v) \cap S =R\}$. Note that the $V_R$ partition  $V(H) \setminus S$. Since there are only $O(1)$ subsets $R \subseteq S$, for some subsets $D,E \subseteq S$, there are $\Omega(s)$ sets $T \in \mathcal{F}$ of the form $T=S \cup \{v,w\}$ for some vertices $v \in V_D$ and $w \in V_E$, as before. \newline

Again, by condition 1 in Definition \ref{r system}, there are only $O(1)$ almost conical vertices in $H[V_D]$ and $H[V_E]$. Hence, without loss of generality, there is a set $T \in \mathcal{F}$ of the form $T=S \cup \{v,w\}$ for some not almost conical vertex $v \in H[V_D]$ and vertex $w \in V_E$. Then there is a vertex $x \in V_D \setminus \left(\Gamma(v) \cup \{v,w\}\right)$, so by condition 2 in Definition \ref{r system}, $H[T]$ is $K_{r-1}$-free but $H[T \cup \{x\}]$ contains $K_{r-1}$. Since $x$ is not adjacent to $v$, this implies that $H[D \cap E]$ contains $K_{r-3}$. \newline

Hence, by condition 1 in Definition \ref{r system}, $H[V_D \cup V_E]$ is $K_3$-free, and by condition 2 in Definition \ref{r system}, $\{v,w\}$ is an independent set for all vertices $v \in V_D$ and $w \in V_E$ with $S \cup \{v,w\} \in \mathcal{F}$. Finally, we show that all such independent sets $\{v,w\}$ are in fact maximally independent in $H[V_D \cup V_E]$, so that $\left(H[V_D \cup V_E],\mathcal{F}'\right)$ is the desired $(3,2)'$-system, where $\mathcal{F}'$ is the family of all such pairs $\{v,w\}$. Let $v \in V_D$ and $w \in V_E$ be such that $S \cup \{v,w\} \in \mathcal{F}$ and $x \in \left(V_D \cup V_E\right) \setminus \{v,w\}$. We need to show that $x$ is adjacent to $v$ or $w$. By condition 2 in Definition \ref{r system}, $H\left[S \cup \{v,w\}\right]$ is $K_{r-1}$-free but $H\left[S \cup \{v,w,x\}\right]$ contains $K_{r-1}$, which implies $x$ is adjacent to $v$ or $w$, for otherwise $\Gamma(x) \cap \left(S \cup \{v,w\}\right) \subseteq \Gamma(v) \cap S$ or $\Gamma(x) \cap \left(S \cup \{v,w\}\right) \subseteq \Gamma(w) \cap S$.

\end{proof}

\vspace{1cm}

We will use the following estimate for $s_{3,3}'(m)$ in the proofs of the lower and upper bounds in part 4 of Theorem \ref{set estimates} and the lower bounds in parts 2 and 5 of Theorem \ref{edge estimates}. Since it is of independent interest, we state it as a theorem. 

\begin{theorem}\label{many cases}
$s_{3,3}'(m)=m+O(1)$.
\end{theorem}

\begin{proof}

By properties 1, 2 and 4 in Construction \ref{system construction}, $s_{3,3}'(2l) \geq 2l$ for large enough integers $l$, so $s_{3,3}'(m) \geq m+O(1)$ by Lemma \ref{3 prime approx incr}. We now prove that $s_{3,3}'(m) \leq m+O(1)$. Let $(H,\mathcal{F})$ be a $(3,3)'$-system with $|H|=m$ and $|\mathcal{F}|=s$. We need to show that $s \leq m+O(1)$. We distinguish four cases. \newline

\emph{Case 1:} There is a set $S \in \mathcal{F}$ and for each vertex $v \in S$ there is a set $S_v \in \mathcal{F}$ such that $S \cap S_v=\{v\}$ and the $S_v$ are disjoint.

\begin{centering}
     
\begin{tikzpicture}[scale=0.3]
      
             \filldraw[black] ($(0,0)+(90:2)$) circle (3pt) ;
             \filldraw[black] ($(0,0)+(210:2)$) circle (3pt) ;
             \filldraw[black] ($(0,0)+(330:2)$) circle (3pt) ;
             \draw [black] plot [smooth cycle, tension=0.8] coordinates {($(0,0)+(90:2.5)$)  ($(0,0)+(210:2.5)$)  ($(0,0)+(330:2.5)$)};    
             \filldraw[black] ($(0,4)+(30:2)$) circle (3pt) ;
             \filldraw[black] ($(0,4)+(150:2)$) circle (3pt) ;
             \filldraw[black] ($(0,0)+(210:4)+(150:2)$) circle (3pt) ;
             \filldraw[black] ($(0,0)+(210:4)+(270:2)$) circle (3pt) ;
             \filldraw[black] ($(0,0)+(330:4)+(270:2)$) circle (3pt) ;
             \filldraw[black] ($(0,0)+(330:4)+(30:2)$) circle (3pt) ;
             \draw [black] plot [smooth cycle, tension=0.8] coordinates {($(0,0)+(90:1.5)$) ($(0,4)+(30:2.5)$) ($(0,4)+(150:2.5)$)}; 
             \draw [black] plot [smooth cycle, tension=0.8] coordinates {($(0,0)+(210:1.5)$) ($(210:4)+(150:2.5)$) ($(210:4)+(270:2.5)$)}; 
             \draw [black] plot [smooth cycle, tension=0.8] coordinates {($(0,0)+(330:1.5)$) ($(330:4)+(270:2.5)$) ($(330:4)+(30:2.5)$)}; 
            
\end{tikzpicture}


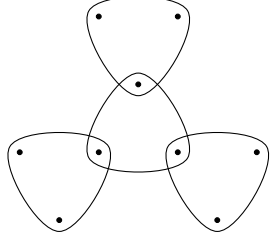
\captionof{figure}{The set $S$ and the sets $S_v$.}

\end{centering}

\vspace{1cm}   

We show that every vertex $v \in S$ is adjacent to all vertices except for those in $S \cup S_v$ and that there are no other edges in $H$ and sets in $\mathcal{F}$. Hence $s=4$. \newline

Let $S=\{a,b,c\}$. By condition 2 in Definition \ref{3 prime system definition}, every vertex  in the induced subgraph between $\{b,c\}$ and $S_a \setminus \{a\}$ has degree at least one. It is easy to check that every bipartite graph with two vertices in each part in which every vertex has degree at least one contains a perfect matching. So let $S_a \setminus \{a\}=\{(a,b),(a,c)\}$, where $(a,b)$ and $(a,c)$ are adjacent to $b$ and $c$, respectively. Similarly, let $S_b \setminus \{b\}=\{(b,a),(b,c)\}$ and $S_c \setminus \{c\}=\{(c,a),(c,b)\}$, where $(b,a)$, $(b,c)$, $(c,a)$ and $(c,b)$ are adjacent to $a$, $c$, $a$ and $b$, respectively. \newline

Since $(a,b)$ is adjacent to $b$, which is adjacent to $(c,b)$, by condition 1 in Definition \ref{3 prime system definition}, $(a,b)$ is not adjacent to $(c,b)$. Similarly, $(a,c)$ is not adjacent to $(b,c)$ and $(b,a)$ is not adjacent to $(c,a)$. Suppose for the sake of contradiction that $(a,b)$ is not adjacent to $c$. Then, since $(a,b)$ is also not adjacent to $(c,b)$, by condition 2 in Definition \ref{3 prime system definition}, $(a,b)$ must be adjacent to $(c,a)$. Hence, since $(a,b)$ is adjacent to $b$, by condition 1 in Definition \ref{3 prime system definition}, $(c,a)$ is not adjacent to $b$. \newline

Therefore, since $(c,a)$ is also not adjacent to $(b,a)$, by condition 2 in Definition \ref{3 prime system definition}, $(c,a)$ must be adjacent to $(b,c)$. Then, since $(c,a)$ is adjacent to $a$, by condition 1 in Definition \ref{3 prime system definition}, $(b,c)$ is not adjacent to $a$. Hence, since $(b,c)$ is also not adjacent to $(a,c)$, by condition 2 in Definition \ref{3 prime system definition}, $(b,c)$ must be adjacent to $(a,b)$. But then $(a,b)$, $(b,c)$ and $(c,a)$ form a triangle, contradicting condition 1 in Definition \ref{3 prime system definition}. \newline

Hence, $(a,b)$ is adjacent to $c$, so by symmetry, every vertex $v \in S$ is adjacent to every vertex in $U \setminus \left(S \cup S_v \right)$, where $U=S \cup S_a \cup S_b \cup S_c$. Then by conditions 1 and 2 in Definition \ref{3 prime system definition}, there are no other edges in $H[U]$. Let $x \in V(H) \setminus U$. Then by condition 1 in Definition \ref{3 prime system definition}, $\Gamma(x)$ is an independent set, and by condition 2 in Definition \ref{3 prime system definition}, $\Gamma(x) \cap S \neq \emptyset$ and $\Gamma(x) \cap S_v \neq \emptyset$ for all $v \in S$, which implies $\Gamma(x) \cap U=S$. Hence $V(H) \setminus U$ is an independent set by condition 1 in Definition \ref{3 prime system definition}, so every vertex $v \in S$ is adjacent to all vertices except for those in $S \cup S_v$ and there are no other edges in $H$. Finally, it is easy to check that the only maximally independent sets of size three in $H$ are $S$ and the $S_v$, so there are no other sets in $\mathcal{F}$ by condition 2 in Definition \ref{3 prime system definition}. \newline

\emph{Case 2:} There is a set $S \subseteq V(H)$ of size three and for each vertex $v \in S$ there is a vertex $v' \not \in S$ such that $S_v:=\left(S \setminus \{v\} \right)\cup \{v'\} \in \mathcal{F}$ and the $v'$ are distinct. \newline

\begin{centering}
     
\begin{tikzpicture}[scale=0.75]
      
             \filldraw[black] ($(0,0)+(90:1)$) circle (1.2pt) ;
             \filldraw[black] ($(0,0)+(210:1)$) circle (1.2pt) ;
             \filldraw[black] ($(0,0)+(330:1)$) circle (1.2pt) ;
             \filldraw[black] ($(0,0)+(150:2)$) circle (1.2pt) ;
             \filldraw[black] ($(0,0)+(270:2)$) circle (1.2pt) ;
             \filldraw[black] ($(0,0)+(30:2)$) circle (1.2pt) ;
             \draw [black] plot [smooth cycle, tension=0.8] coordinates {($(0,0)+(150:2.25)$) ($(0,0)+(90:1)+(30:0.25)$) ($(0,0)+(210:1)+(270:0.25)$)};
             \draw [black] plot [smooth cycle, tension=0.8] coordinates {($(0,0)+(270:2.25)$) ($(0,0)+(210:1)+(150:0.25)$) ($(0,0)+(330:1)+(30:0.25)$)};
             \draw [black] plot [smooth cycle, tension=0.8] coordinates {($(0,0)+(30:2.25)$) ($(0,0)+(330:1)+(270:0.25)$) ($(0,0)+(90:1)+(150:0.25)$)};

\end{tikzpicture}


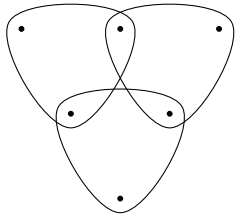
\captionof{figure}{The sets $S_v$.}

\end{centering}

\vspace{1cm}   

By condition 2 in Definition \ref{3 prime system definition}, $v$ is adjacent to $v'$ for all $v \in S$. Let $x \in V(H) \setminus U$, where $U=S \cup S'$ and $S'=\bigcup_{v \in S} \{v'\}$. Then by condition 1 in Definition \ref{3 prime system definition}, $\Gamma(x)$ is an independent set, and by condition 2 in Definition \ref{3 prime system definition}, $\Gamma(x) \cap S_v \neq \emptyset$ for all $v \in S$, which implies that either $\Gamma(x) \cap U=S'$ or $|\Gamma(x) \cap S| \geq 2$. Hence we have a partition $V(H) \setminus S'=X \cup Y$, where $X=\{x \in V(H) \setminus U: \Gamma(x) \cap U=S'\} \cup S$ and $Y=\{x \in V(H) \setminus U: |\Gamma(x) \cap S| \geq 2\}$ are independent sets by conditions 1 and 2 in Definition \ref{3 prime system definition}. \newline

Recall from the proof of Lemma \ref{3 3 stability} that, since a $(3,1)'$-system is a family of conical vertices in a $K_3$-free graph, $s_{3,1}'(m) \leq 2$ for all $m$. For every pair $P$ of vertices in $H$, if we let $W=V(H) \setminus \left[\bigcup_{v \in P} \left( \Gamma(v) \cup \{v\} \right) \right]$ and $\mathcal{F}'=\{T \setminus P : P \subseteq T \in \mathcal{F} \}$, then applying the first part of Observation \ref{induction observation} twice, we obtain that $\left(H[W],\mathcal{F}'\right)$ is a $(3,1)'$-system, so there are at most $2$ sets $T \in \mathcal{F}$ with $P \subseteq T$. \newline

Since there are only $O(1)$ pairs $P \subseteq U$, it follows that there are only $O(1)$ sets $T \in \mathcal{F}$ with $|T \cap U| \geq 2$. Since $X$ is an independent set, by condition 2 in Definition \ref{3 prime system definition}, there is at most one set $T \in \mathcal{F}$ with $T \subseteq X$. Similarly, there is at most one set $T \in \mathcal{F}$ with $T \subseteq Y$. \newline

Suppose $I$ is an independent set, $u \not \in I$, $v \neq w \in I$ and $\{u,v,w\} \in \mathcal{F}$. Then by condition 2 in Definition \ref{3 prime system definition}, $v$ and $w$ are the only vertices in $I$ not adjacent to $u$. It follows that for every vertex $u$ and independent set $I$ with $u \not \in I$ there is at most one set in $\mathcal{F}$ of the form 
$\{u,v,w\}$ for some vertices $v \neq w \in I$. \newline

For every vertex $u \in S'$, since $u$ is adjacent to every vertex in $X \setminus S$, by condition 2 in Definition \ref{3 prime system definition}, every set $T \in \mathcal{F}$ with $T \cap U=\{u\}$ is of the form $T=\{u,v,w\}$ for some vertices $v \neq w \in Y$. Since $Y$ is an independent set, it follows that there are only $O(1)$ sets $T \in \mathcal{F}$ with $T \cap U=\{u\}$ for some $u \in S'$. Also, for every vertex $x \in X$, since $Y$ is an independent set, there is at most one set in $\mathcal{F}$ of the form $\{x,y,z\}$ for some vertices $y \neq z \in Y$. Hence, there are at most $|X|$ sets $T \in \mathcal{F}$ with $|T \cap X|=1$ and $|T \cap Y|=2$. Similarly, there are at most $|Y|$ sets $T \in \mathcal{F}$ with $|T \cap X|=2$ and $|T \cap Y|=1$, so there are at most $|X|+|Y|=m-3$ sets $T \in \mathcal{F}$ with $T \subseteq X \cup Y = V(H) \setminus S'$ and $T \not \subseteq X, Y$. \newline

Putting everything together, we obtain $s \leq m+O(1)$. \newline

\emph{Case 3:} There are sets $S,T \in \mathcal{F}$ with $|S \cap T|=2$. \newline

\begin{centering}
     
\begin{tikzpicture}[scale=0.6]
      
             \filldraw[black] (0,1) circle (1.2pt) ;
             \filldraw[black] (0,-1) circle (1.2pt) ;
             \filldraw[black] (-1.732,0) circle (1.2pt) ;
             \filldraw[black] (1.732,0) circle (1.2pt) ;
             \draw [black] plot [smooth cycle, tension=0.8] coordinates {(0.1,1.25) (0.1,-1.25) (-2,0)};
            \draw [black] plot [smooth cycle, tension=0.8] coordinates {(-0.1,1.25) (-0.1,-1.25) (2,0)};

\end{tikzpicture}


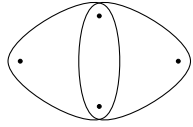
\captionof{figure}{The sets $S$ and $T$.}

\end{centering}

\vspace{1cm} 

Let $S=\{a,b,c\}$ and $T=\{a,b,d\}$. Then by condition 2 in Definition \ref{3 prime system definition}, $c$ is adjacent to $d$. Hence, for every vertex $v \in V(H) \setminus \{a,b,c,d\}$, by condition 1 in Definition \ref{3 prime system definition}, $v$ cannot be adjacent to both $c$ and $d$, so by condition 2 in Definition \ref{3 prime system definition}, $v$ must be adjacent to $a$ or $b$. For every such vertex $v$, pick a vertex in $\{a,b\}$ that $v$ is adjacent to and let $X$ and $Y$ be the sets of vertices $v$ for which $a$ and $b$ were picked, respectively. Then by condition 1 in Definition \ref{3 prime system definition}, $X$ and $Y$ are independent sets. Note that $X$ and $Y$ partition $V(H) \setminus \{a,b,c,d\}$. \newline

Recall from Case 2 that for every pair $P$ of vertices in $H$, there are at most 2 sets $R \in \mathcal{F}$ with $P \subseteq R$. Since there are only $O(1)$ pairs $P \subseteq \{a,b,c,d\}$, it  follows that there are only $O(1)$ sets $R \in \mathcal{F}$ with $|R \cap \{a,b,c,d\}| \geq 2$. Since $X$ is an independent set, by condition 2 in Definition \ref{3 prime system definition}, there is at most one set $R \in \mathcal{F}$ with $R \subseteq X$, as before. Similarly, there is at most one set $R \in \mathcal{F}$ with $R \subseteq Y$. \newline

Recall also from Case 2 that for every vertex $u$ and independent set $I$ with $u \not \in I$ there is at most one set in $\mathcal{F}$ of the form $\{u,v,w\}$ for some vertices $v \neq w \in I$. Hence, there is at most one set in $\mathcal{F}$ of the form $\{c,v,w\}$ for some vertices $v \neq w \in X$, $\{c,v,w\}$ for some vertices $v \neq w \in Y$, $\{d,v,w\}$ for some vertices $v \neq w \in X$ and $\{d,v,w\}$ for some vertices $v \neq w \in Y$. Since $a$ is adjacent to every vertex in $X$, by condition 2 in Definition \ref{3 prime system definition}, every set $R \in \mathcal{F}$ with $R \cap \{a,b,c,d\}=\{a\}$ is of the form $R=\{a,v,w\}$ for some vertices $v \neq w \in Y$. Since $Y$ is an independent set, it follows that there is at most one set $R \in \mathcal{F}$ with $R \cap \{a,b,c,d\}=\{a\}$. Similarly, there is at most one set $R \in \mathcal{F}$ with $R \cap \{a,b,c,d\}=\{b\}$. \newline

For every vertex $x \in X$, since $Y$ is an independent set, there is at most one set in $\mathcal{F}$ of the form $\{x,y,z\}$ for some vertices $y \neq z \in Y$. Let $X' \subseteq X$ be the set of vertices $x \in X$ for which there is such a set. Then there are $|X'|$ sets $R \in \mathcal{F}$ with $|R \cap X|=1$ and $|R \cap Y|=2$. Similarly, letting $Y' \subseteq Y$ be the set of vertices $y \in Y$ for which there is a set $\{y,x,z\} \in \mathcal{F}$ for some vertices $x \neq z \in X$, there are $|Y'|$ sets $R \in \mathcal{F}$ with $|R \cap X|=2$ and $|R \cap Y|=1$, so there are $|X'|+|Y'|=|X' \cup Y'|$ sets $R \in \mathcal{F}$ with $R \subseteq X \cup Y=V(H) \setminus \{a,b,c,d\}$ and $R \not \subseteq X, Y$. \newline

Putting everything together, we obtain $s \leq |X' \cup Y'|+|\mathcal{F}'|+O(1)$, where $\mathcal{F}' \subseteq \mathcal{F}$ is the family of sets in $\mathcal{F}$ of the form $\{v,x,y\}$ for some vertices $v \in \{c,d\}$, $x \in X$ and $y \in Y$. So it suffices to show that $|\mathcal{F}'| \leq |(X \cup Y) \setminus (X' \cup Y')|$. We do this by defining an injection $i : \mathcal{F}' \to (X \cup Y) \setminus (X' \cup Y')$. Let $i\left(\{v,x,y\}\right)=x$ if either $v=c$ and $x \not \in X'$ or $v=d$ and $y \in Y'$ and let $i\left(\{v,x,y\}\right)=y$ if either $v=c$ and $x \in X'$ or $v=d$ and $y \not \in Y'$. \newline

We need to show that for all $\{v,x,y\} \in \mathcal{F}'$, we cannot have both $x \in X'$ and $y \in Y'$, so that $i$ indeed maps into $(X \cup Y) \setminus (X' \cup Y')$. Suppose for the sake of contradiction that $x \in X'$ and $y \in Y'$ for some $\{v,x,y\} \in \mathcal{F}'$. Let $Q$ and $R$ be the unique sets in $\mathcal{F}$ with $Q \cap X=\{x\}$, $|Q \cap Y|=2$, $R \cap Y=\{y\}$ and $|R \cap X|=2$. By condition 2 in Definition \ref{3 prime system definition}, $x$ and $y$ are not adjacent. Hence, since $Q$ consists of $x$ and the two vertices in $Y$ not adjacent to $x$, we must have $y \in Q$. Similarly, $x \in R$. But then $\{v,x,y\}$, $Q$ and $R$ are three distinct sets in $\mathcal{F}$ containing the pair $\{x,y\}$, a contradiction. \newline

Finally, we show that $i$ is injective. Suppose for the sake of contradiction that $i\left(\{v,x,y\}\right)=i\left(\{v',x',y'\}\right)$ for some $\{v,x,y\} \neq \{v',x',y'\} \in \mathcal{F}'$. Without loss of generality, suppose $i\left(\{v,x,y\}\right)=i\left(\{v',x',y'\}\right) \in X$. Then $x=i\left(\{v,x,y\}\right)=i\left(\{v',x',y'\}\right)=x'$, either $v=c$ or $v=d$ and $y \in Y'$, and either $v'=c$ or $v'=d$ and $y' \in Y'$. Suppose first that $v=v'$. Then $y \neq y'$ and $\{v,x,y\}, \{v,x,y'\} \in \mathcal{F}$, so by condition 2 in Definition \ref{3 prime system definition}, $y$ is adjacent to $y'$. This is a contradiction, since $Y$ is an independent set. \newline

Now suppose $v \neq v'$. Then without loss of generality, $v=c$, $v'=d$ and $y' \in Y'$. Let $R$ be the unique set in $\mathcal{F}$ with $R \cap Y=\{y'\}$ and $|R \cap X|=2$. Since $\{d,x,y'\} \in \mathcal{F}$, by condition 2 in Definition \ref{3 prime system definition}, $x$ is not adjacent to $y'$. Since $R$ consists of $y'$ and the two vertices in $X$ not adjacent to $y'$, we must have $R=\{y',x,x'\}$ for some vertex $x' \in X \setminus \{x\}$. Then $\{d,x,y'\}, \{y',x,x'\} \in \mathcal{F}$, so by condition 2 in Definition \ref{3 prime system definition}, $d$ and $x'$ are adjacent. \newline

Hence, for every vertex $w \in V(H) \setminus \{d,x,x',y'\}$, by condition 1 in Definition \ref{3 prime system definition}, $w$ cannot be adjacent to both $d$ and $x'$, so by condition 2 in Definition \ref{3 prime system definition}, $w$ must be adjacent to $x$ or $y'$. In particular, $c$ must be adjacent to $x$ or $y'$. But $\{c,x,y\}$ is in $\mathcal{F}$ and therefore an independent set by condition 2 in Definition \ref{3 prime system definition}, so $c$ must be adjacent to $y'$, which implies $y \neq y'$. Then $y \in V(H) \setminus \{d,x,x',y'\}$, so $y$ must be adjacent to $x$ or $y'$, which is impossible, since $\{c,x,y\}$ and $Y$ are independent sets. \newline

\emph{Case 4:} Cases 1, 2 and 3 do not hold. \newline

We show that for every set $S \in \mathcal{F}$ there is a vertex $v \in S$ with $s(v)=1$. Then picking such a vertex for every set in $\mathcal{F}$ gives an injection $\mathcal{F} \to V(H)$, so $s \leq m$. Suppose for the sake of contradiction that there is a set $S \in \mathcal{F}$ with $s(v) \geq 2$ for all $v \in S$. Then for every $v \in S$, there is a set $S_v \in \mathcal{F} \setminus \{S\}$ with $v \in S_v$. Since Case 3 does not hold, $S \cap S_v=\{v\}$ for all $v \in S$. Since Cases 2 and 3 do not hold, the $S_v$ are disjoint. But then Case 1 holds, which is a contradiction.

\end{proof}

\begin{rem}
The similarities between the proofs in the different cases perhaps suggest that the number of cases in the proof can be reduced, but the author has not found a way of doing this. 
\end{rem}

\vspace{1cm}

Finally, we prove the following lemma, which we will use in the proof of the lower bound in part 2 of Theorem \ref{edge estimates}.

\begin{lemma}\label{3 3 edge result}
Let $(H,\mathcal{F})$ be a $(3,3)'$-system with $|H|=m$, $e(H)=e$ and $|\mathcal{F}|=s$ large. Then $e \gg m^{1/2}s$. 
\end{lemma}

\begin{proof}

Note that $m$ is large, since $s$ is large, and that $e \geq m-3 \gg m$, since by condition 2 in Definition \ref{3 prime system definition}, every vertex outside a set $S \in \mathcal{F}$ is adjacent to a vertex in $S$. Since $\sum_{v \in H}d(v)=2e$, there is a vertex $v$ with $d(v) \ll e/m$. Then by condition 2 in Definition \ref{3 prime system definition}, every set in $\mathcal{F}$ contains a vertex in $\Gamma(v) \cup \{v\}$, so by a union bound, $s \leq \sum_{w \in \Gamma(v) \cup \{v\}} s(w)$. Let $W=\{w \in \Gamma(v) \cup \{v\} : s(w)=\omega(1) \}$. Then $|\Gamma(v) \cup \{v\}| = d(v)+1 \ll e/m$, so $|W| \ll e/m$ and $\sum_{w \in W} s(w) \geq s-O(e/m)$. If $e \gg ms$, there is nothing to prove, so we may assume $e=o(ms)$, so that $\sum_{w \in W} s(w) \gg s$. \newline

Let $E' \subseteq E(H)$ be the set of edges $xy$ with $s(x)=O(1)$ or $s(y)=O(1)$, $e'=|E'|$ and $Z$ be the set of pairs $(w,xy)$, where $w \in W$ and $xy \in E'$, for which there are sets $S,T \in \mathcal{F}$ with $\{w,x\} \subseteq S$ and $\{w,y\} \subseteq T$. We double count the number of pairs in $Z$. On the one hand, for all $xy \in E'$, without loss of generality $s(x)=O(1)$, so there are only $O(1)$ sets $S \in \mathcal{F}$ with $\{w,x\} \subseteq S$ and hence only $O(1)$ vertices $w$ with $(w,xy) \in Z$. Hence $|Z| \ll e'$. \newline

On the other hand, for each $w \in W$, let $V_w \subseteq V(H)$ be the set of vertices $x$ for which there is a set $S \in \mathcal{F}$ with $\{w,x\} \subseteq S$ and let $\mathcal{F}_w=\{S \in \mathcal{F} : w \in S\}$. Then it is easy to see that, since $(H,\mathcal{F})$ is a $(3,3)'$-system, $(H[V_w],\mathcal{F}_w)$ is a $(3,3)$-system. Since $s(w)=\omega(1)$, by Lemma \ref{3 3 stability}, $(H[V_w],\mathcal{F}_w)=(H_{2,l}',\mathcal{F}_{2,l}')$ for some even $l$. By property 4 in Construction \ref{system construction}, $l=2s(w)$. Since $w$ is in all the sets in $\mathcal{F}_w$ and the isolated vertex of $H_{2,l}'$ is the only vertex in all the sets in $\mathcal{F}_{2,l}'$, $w$ must be the isolated vetex of $H_{2,l}'$. \newline

Recall that in Case 4 of the proof of Theorem \ref{many cases} we showed that if Cases 1, 2 and 3 do not hold then for every set $S \in \mathcal{F}$ there is a vertex $v \in S$ with $s(v)=1$. We now show that it is always the case that for every set $S \in \mathcal{F}$ there is a vertex $v \in S$ with $s(v)=O(1)$. Suppose for the sake of contradiction that there is a set $S \in \mathcal{F}$ with $s(v)=\omega(1)$ for all $v \in S$. Recall from the proof of Theorem \ref{many cases} that for every pair $P$ of vertices in $H$, there are at most $2$ sets $T \in \mathcal{F}$ with $P \subseteq T$. \newline

Let $S=\{a,b,c\}$. Since $s(a)=\omega(1)$ and there are at most $2$ sets $T \in \mathcal{F}$ with $\{a,b\} \subseteq T$ and at most $2$ sets $T \in \mathcal{F}$ with $\{a,c\} \subseteq T$, there is a set $S_a \in \mathcal{F}$ with $S_a \cap S=\{a\}$. Then, since $s(b)=\omega(1)$ and for every vertex $v \in \left(S \cup S_a\right) \setminus \{b\}$ there are at most $2$ sets $T \in \mathcal{F}$ with $\{b,v\} \subseteq T$, there is a set $S_b \in \mathcal{F}$ with $S_b \cap \left(S \cup S_a\right)=\{b\}$. Similarly, since $s(c)=\omega(1)$, there is a set $S_c \in \mathcal{F}$ with $S_c \cap \left(S \cup S_a \cup S_b\right)=\{c\}$. But then $S$, $S_a$, $S_b$ and $S_c$ are as in Case 1 of the proof of Theorem \ref{many cases}, so $s=4$, contradicting $s(a),s(b),s(c)=\omega(1)$. \newline

Hence, for all matched vertices $x$ and $y$ in $H_{2,l}$, since $\{w,x,y\} \in \mathcal{F}$ and $s(w)=\omega(1)$, either $s(x)=O(1)$ or $s(y)=O(1)$. Let $X$ and $Y$ be the two parts of $H_{2,l}$. It follows that for all $x \neq x' \in X$ and $y \neq y' \in Y$ with $x$ and $y$ matched and $x'$ and $y'$ matched, either $xy' \in E'$ or $x'y \in E'$. Hence, since $l=2s(w)$, at least $\binom{s(w)}{2}$ edges in $H_{2,l}$ are in $E'$, so there are $\Omega(s(w)^2)$ edges $xy$ with $(w,xy) \in Z$. Hence $|Z| \gg \sum_{w \in W} s(w)^2$.  \newline

Putting everything together, we obtain

$$e \ \geq \ e' \ \gg \ |Z| \ \gg \ \sum_{w \in W} s(w)^2 \ \gg \ \frac{s^2}{e/m} $$

by the Cauchy-Schwarz inequality, since $\sum_{w \in W} s(w) \gg s$ and $|W| \ll e/m$. This implies $e \gg m^{1/2}s$. \newline
\end{proof}

\begin{rem}

In the proof we showed that if $(H,\mathcal{F})$ is a $(3,3)'$-system then for every set $S \in \mathcal{F}$ there is a vertex $v \in S$ with $s(v)=O(1)$. Picking such a vertex for every set in $\mathcal{F}$ gives a function $\mathcal{F} \to V(H)$ such that the inverse image of every vertex has size $O(1)$. This gives a proof of the estimate $s_{3,3}'(m)=\Theta(m)$ that is shorter and simpler than that of Theorem \ref{many cases}, only needing to consider Case 1. For all of our applications of Theorem \ref{many cases}, this weaker estimate is sufficient, except for in the proof of the lower bound in part 5 of Theoerem \ref{edge estimates}, where we will need the precise constant factor.

\end{rem}

\vspace{1cm}

We are now ready to prove Theorems \ref{set estimates} and \ref{edge estimates}. We first prove Theorem \ref{set estimates}, which we will use in the proofs of the lower and upper bounds in part 3 of Theorem \ref{edge estimates}, the lower bounds in parts 1 and 2 of Theorem \ref{twin free result}, and the lower bounds in parts 3 and 4 of Theorem \ref{twin free min deg result}.

\begin{repth11}

We have the following estimates for $s_{r,t}(m)$.

\begin{enumerate}
\item For all integers $r \geq 3$, $s_{r,r-2}(m)=1$ for large enough $m$.
\item For all integers $r \geq 3$, $s_{r,r-1}(m)=2$ for large enough $m$.
\item For all integers $r \geq 3$, $s_{r,r}(m)=\lfloor (m-r+2)/2 \rfloor$ for large enough $m$.
\item For all integers $r \geq 3$, $s_{r,r+1}(m)=\Theta(m)$.
\item For all integers $r \geq 3$ and $t \geq r+2$,
$$m^{(t-r+2)/2} \ll s_{r,t}(m) \ll m^{t-r+2} \ . $$ 
\item For all integers $t \geq 5$,
$$m^{(t-1)/2} \ll s_{3,t}(m) \ll m^{t-3} \ . $$
\end{enumerate}

\end{repth11}

\begin{proof}

We first prove part 1. By Observation \ref{obs 1}, $s_{r,r-2}(m)$ is the maximum size of an independent set of vertices of degree $r-2$ with distinct neighbourhoods in a $K_r$-saturated graph whose complement has size $m$. It is easy to check that the only $K_r$-saturated graphs with a vertex of degree $r-2$ are the ones obtained by adding $r-2$ conical vertices to a graph with no edges, from which it easily follows that $s_{r,r-2}(m)=1$ for $m \geq r-2$. \newline

We now prove part 2. By Observation \ref{obs 1} again, $s_{r,r-1}(m)$ is the maximum size of an independent set of vertices of degree $r-1$ with distinct neighbourhoods in a $K_r$-saturated graph whose complement has size $m$. It is easy to check that every $K_r$-saturated graph with a vertex of degree $r-1$ is a blow-up of either $K_{r-1}$ or $C_5^{r-3}$, from which it easily follows that $s_{r,r-1}(m)=2$ for $m \geq r$. \newline

Next, we prove part 3. By Lemma \ref{lemma conical systems sets}, the stronger version of Theorem \ref{set estimates 3 t to t+1 relation} in this section, and properties 1, 2 and 4 in Construction \ref{system construction}, 

$$ s_{r,r}(m) \geq s_{3,3}(m-r+3) \geq s_{3,2}'(m-r+2) \geq \lfloor (m-r+2)/2 \rfloor $$
 
for large enough $m$. We now prove the upper bound. Let $(H,\mathcal{F})$ be an $(r,r)$-system with $|H|=m$ large and $|\mathcal{F}|=s$. We need to show that $s \leq \lfloor (m-r+2)/2 \rfloor$. If $s=O(1)$, the inequality holds, since $m$ is large, so we may assume $s=\omega(1)$.
Then by Lemma \ref{conical stability}, there is a set $C$ of $r-3$ conical vertices in $H$, which implies that $(H,\mathcal{F})=(\left(H \setminus C\right)^{r-3},\mathcal{G}^{r-3})$ and $(H \setminus C,\mathcal{G})$ is a $(3,3)$-system, where $\mathcal{G}=\{S \setminus C : S \in \mathcal{F}\}$, since $m$ is large (see section \ref{conical vertices second section}). Let $A=\{v \in V(H) \setminus C: s(v) \geq 1\}$. Then by Lemma \ref{3 3 stability}, $(H[A],\mathcal{G})=(H_{2,l}',\mathcal{F}_{2,l}')$ for some even $l$, so 
$s= l/2 \leq \lfloor (m-r+2)/2 \rfloor$ by properties 2 and 4 in Construction \ref{system construction}. \newline

We now prove part 4. By Lemma \ref{lemma conical systems sets}, the stronger version of Theorem \ref{set estimates 3 t to t+1 relation} and Theorem \ref{many cases},
$$s_{r,r+1}(m) \geq s_{3,4}(m-r+3) \geq s_{3,3}'(m-r+2) = m+O(1) \ .$$
We now prove the upper bound. Let $(H,\mathcal{F})$ be an $(r,r+1)$-system with $|H|=m$ and $|\mathcal{F}|=s$. We need to show that $s \ll m$. By Lemma \ref{r r+1 stability}, there is a $(3,2)'$-system or $(3,3)'$-system $(H',\mathcal{F}')$ with $H' \subseteq H$ and $|\mathcal{F}'|=\Omega(s)$. Let $m'=|H'|$ and $s'=|\mathcal{F}'|$. Then $m' \leq m$ and $s' \gg s$. By Theorem \ref{many cases} and the stronger version of Theorem \ref{set estimates 3 t to t+1 relation}, $s_{3,3}'(n)=n+O(1)$ and $s_{3,2}'(n) \leq s_{3,3}(n+1) \leq s_{3,3}'(n+1)=n+O(1)$. Hence $s \ll s' \leq m'+O(1) \leq m+O(1)$.  \newline

Next, we prove part 5. The lower bound follows from Lemma \ref{lemma conical systems sets} and  the lower bound in part 6. We now prove the upper bound. Let $(H,\mathcal{F})$ be an $(r,t)$-system with $|H|=m$ large. We need to show that $|\mathcal{F}| \ll m^{t-r+2}$. A celebrated result of Erd\H{o}s, Ko and Rado (Theorem 2 (b) in \cite{Erdos Ko Rado}) states that for all integers $t \geq k \geq 0$, if $\mathcal{F}$ is a family of subsets of size $t$ of a set of size $m$ such that the intersection of any two distinct sets in the family has size at least $k$, then $|\mathcal{F}| \leq \binom{m-k}{t-k}$ if $m$ is large enough. Hence, by condition 3 in Definition \ref{r system}, $|\mathcal{F}| \leq \binom{m-r+2}{t-r+2} \asymp m^{t-r+2}$. \newline

Finally, we prove part 6. The upper bound follows from part 4 with $r=3$ by iterating the upper bound in Theorem \ref{set estimates 3 t to t+1 relation}. We now prove the lower bound. By the stronger version of Theorem \ref{set estimates 3 t to t+1 relation}, this is equivalent to the statement that $s_{3,t}'(m) \gg m^{t/2}$ for all integers $t \geq 4$. By properties 1, 2 and 4 in Construction \ref{system construction}, for large enough integers $l$, $s'_{3,4}\left(2l^2\right) \geq l^2(l-1)^2/2 $ and $s_{3,t}'\left(tl^2\right) \geq l^t$ for $t \geq 5$, so $s_{3,t}'(m) \gg m^{t/2}$ for $t \geq 4$ by Lemma \ref{3 prime approx incr}. \newline

\end{proof}

\vspace{1cm}

We now prove Theorem \ref{edge estimates}, which we will use in the proofs of the lower and upper bounds in Theorem \ref{connection with edges}, the lower bound in part 2 of Theorem \ref{twin free result}, and the lower bounds in parts 1, 2 and 4 and the upper bounds in parts 1 and 2 of Theorem \ref{twin free min deg result}.

\begin{repth17}
We have the following estimates for $e_{r,t}(s)$ and $e_{r,t}'(s)$.

\begin{enumerate}
\item For all integers $r \geq 3$, $e_{r,r}(s)=e_{r,r}'(s)=s^2+(2r-7)s+\binom{r-2}{2}$ for large enough $s$.
\item For all integers $r \geq 3$, $e_{r,r+1}(s), e_{r,r+1}'(s)=\Theta\left(s^{3/2}\right)$.
\item For all integers $r \geq 3$ and $t \geq r+2$,
$$s^{1/(t-r+2)} \ll e_{r,t}(s) \leq  e'_{r,t}(s)  \ll s^{4/(t-r+2)} \ . $$ 
\item For all integers $t \geq 5$,
$$s^{2/(t-3)} \ll e_{3,t}(s) \leq  e'_{3,t}(s) \ll s^{4/(t-1)} \ . $$
\item We have
$$\frac{2}{3} s+o(s) \leq e_{3,5}(s) \leq e_{3,5}'(s) \leq 2s+o(s) \ . $$
\end{enumerate}
\end{repth17}

\begin{proof}

We first prove part 1. By Lemma \ref{lemma conical systems edges} and properties 2 through 5 in Construction \ref{system construction}, 

$$e_{r,r}'(s) \leq e(H_{2,2s}')+(r-3)|H_{2,2s}'|+\binom{r-3}{2} = s^2+(2r-7)s+\binom{r-2}{2} $$
 
for large enough $s$. We now prove the lower bound. Let $(H,\mathcal{F})$ be an $(r,r)$-system with $|\mathcal{F}|=s$ large. We need to show that $e(H) \geq s^2+(2r-7)s+\binom{r-2}{2}$. Note that $|H|$ is large, since $s$ is large. As in the proof of Theorem \ref{set estimates}, by Lemma \ref{conical stability}, there is a set $C$ of $r-3$ conical vertices in $H$, which implies that $(H,\mathcal{F})=(\left(H \setminus C\right)^{r-3},\mathcal{G}^{r-3})$ and $(H \setminus C,\mathcal{G})$ is a $(3,3)$-system, where $\mathcal{G}=\{S \setminus C : S \in \mathcal{F}\}$, since $|H|$ is large (see section \ref{conical vertices second section}). Let $A=\{v \in V(H) \setminus C: s(v) \geq 1\}$. Then by Lemma \ref{3 3 stability}, $(H[A],\mathcal{G})=(H_{2,l}',\mathcal{F}_{2,l}')$ for some even $l$, as before, so 

$$e(H) = e\left(H \setminus C\right)+(r-3)|H \setminus C|+\binom{r-3}{2} \geq s^2+(2r-7)s+\binom{r-2}{2} $$

by properties 2, 3 and 4 in Construction \ref{system construction}. \newline

We now prove part 2. We first prove the the lower bound. Let $(H,\mathcal{F})$ be an $(r,r+1)$-system with $|\mathcal{F}|=s$. We need to show that $e(H) \gg s^{3/2}$. By Lemma \ref{r r+1 stability}, there is a $(3,2)'$-system or $(3,3)'$-system $(H',\mathcal{F}')$ with $H' \subseteq H$ and $|\mathcal{F}'|=\Omega(s)$. Since $e(H) \geq e(H')$ and $e_{3,2}''(S) \geq e_{3,3}(S) \geq e_{3,3}''(S)$ by the stronger version of Theorem \ref{edge estimates 3 t to t+1 relation} in this section, it suffices to show that $e_{3,3}''(S) \gg S^{3/2}$. Let $(H,\mathcal{F})$ be a $(3,3)'$-system with $|H|=m$, $e(H)=e$ and $|\mathcal{F}|=s$ large. We need to show that $e \gg s^{3/2}$. By Theorem \ref{many cases}, $s \ll m$, and by Lemma \ref{3 3 edge result}, $e \gg m^{1/2}s$, from which $e \gg s^{3/2}$ follows. \newline

We now prove the upper bound. By Corollary \ref{conical lemma edges asymp}, $e_{r,r+1}'(s) \ll e'_{3,4}(s)$, so it suffices to prove the special case $r=3$. Let $s$ be a large integer. We need to construct a maximal $(3,4)$-system $(H,\mathcal{F})$ with $|\mathcal{F}|=s$ and $e(H) \ll s^{3/2}$. Let $l$ be the smallest integer with $\binom{l}{2} \geq s$. Then $l \asymp s^{1/2}$. Let $X$ be a set of size $l$ and $Y \subseteq \binom{X}{2}$ be a set of size $s$, where $\binom{X}{2}$ is the set of all pairs $P \subseteq X$. Let $H'$ be the bipartite graph with parts $X$ and $Y$ where $x \in X$ is adjacent to $P \in Y$ if and only if $x \not \in P$ and $\mathcal{F}'$ be the family of triples consisting of a pair $\{x,y\} \in Y$ and the two elements $x,y \in X$. Then it is easy to check that $(H',\mathcal{F}')$ is a $(3,3)'$-system with $|\mathcal{F}'|=s$ and $e(H') \asymp s^{3/2}$. Let $(H,\mathcal{F})$ be the $(3,4)$-system obtained from $(H',\mathcal{F}')$ using the second part of Observation \ref{induction observation}. Then it is easy to check that $(H,\mathcal{F})$ is in fact a maximal $(3,4)$-system with $|\mathcal{F}|=|\mathcal{F}'|$ and $e(H)=e(H')$. \newline

We now use part 5 of Theorem \ref{set estimates} and Lemma \ref{set to edge lemma} to prove part 3. We first prove the lower bound. Let $(H,\mathcal{F})$ be an $(r,t)$-system with $|H|=m$ and $|\mathcal{F}|=s$. We need to show that $e(H) \gg s^{1/(t-r+2)}$. By the upper bound in part 5 of Theorem \ref{set estimates}, $s \ll m^{t-r+2}$. Hence $e(H) \geq m-t \gg s^{1/(t-r+2)}$ by Lemma \ref{set to edge lemma}. We now prove the upper bound. Let $s$ be a large integer. We need to construct a maximal $(r,t)$-system $(H,\mathcal{F})$ with $|\mathcal{F}|=s$ and $e(H) \ll s^{4/(t-r+2)}$. Let $m$ be the smallest integer with $s_{r,t}(m) \geq s$. By the lower bound in part 5 of Theorem \ref{set estimates}, $m \ll s^{2/(t-r+2)}$. Take an $(r,t)$-system $(H',\mathcal{F}')$ with $|H'|=m$ and $|\mathcal{F}'|=s_{r,t}(m)$ and pick a subfamily $\mathcal{F} \subseteq \mathcal{F}'$ with $\left|\mathcal{F}\right|=s$. Since $(H',\mathcal{F}')$ is an $(r,t)$-system, so is $(H',\mathcal{F})$, so we can add edges to $H'$ to obtain a graph $H$ such that $(H,\mathcal{F})$ is a maximal $(r,t)$-system. Then $e(H) \leq \binom{m}{2} \ll s^{4/(t-r+2)}$ by Lemma \ref{set to edge lemma}. \newline

Next, we prove part 4. We have just proved the upper bound. We prove the lower bound by induction on $t$. The lower bound for $t=5$ follows from the lower bound in part 5. Suppose the lower bound is true for $t \geq 5$. We show it is true for $t+1$. By part 1 of the stronger version of Theorem \ref{edge estimates 3 t to t+1 relation}, this is equivalent to the statement that $e_{3,t}''(s) \gg s^{2/(t-2)}$. Let $(H,\mathcal{F})$ be a $(3,t)'$-system with $|H|=m$, $e(H)=e$ and $|\mathcal{F}|=s$. We need to show that $e \gg s^{(2/t-2)}$. By part 2 of the stronger version of Theorem \ref{edge estimates 3 t to t+1 relation} and the induction hypothesis, $e \geq e_{3,t}(k) \gg k^{2/(t-3)}$ for some $k \geq ms/(2e+m)$. We also trivially have $e \ll m^2$. Combining these gives $e \gg s^{2/(t-2)}$. \newline

Finally, we prove part 5. We first prove the lower bound. Let $(H,\mathcal{F})$ be a $(3,5)$-system with $e(H)=e$ and $|\mathcal{F}|=s$. We need to show that $e \geq 2s/3+o(s)$. By Lemma \ref{principal intersection}, either $(H[A],\mathcal{F})$ can be obtained from a $(3,4)'$-system   $(H',\mathcal{F}')$ using the second part of Observation \ref{induction observation} or there exists a $(3,3)'$-system $(H',\mathcal{F}')$ with $H' \subseteq H$ and $|\mathcal{F}'|=\Omega(s)$, where $A=\{v \in H : s(v) \geq 1\}$. In the second case, $e \geq e(H') \geq e''_{3,3}\left[ \Omega(s) \right] \gg s^{3/2}$ by the proof of part 2. So it suffices to show that $e''_{3,4}(s) \geq 2s/3+o(s)$. \newline

Let $(H,\mathcal{F})$ be a $(3,4)'$-system with $|H|=m$, $e(H)=e$ and $|\mathcal{F}|=s$. We need to show that $e \geq 2s/3+o(s)$. Note that $m \to \infty$ as $s \to \infty$. We double count the number of pairs $(e,S)$, where $e \in E(H)$ and $S \in \mathcal{F}$, with $|e \cap S|=1$. Note that we always have $|e \cap S| \leq 1$ by condition 2 in Definition \ref{3 prime system definition}. On the one hand, for every vertex $v \in H$, the number of pairs with $e \cap S=\{v\}$ is $d(v)s(v)$, so there are $\sum_{v \in H} d(v)s(v)$ such pairs. On the other hand, as in the proof of Lemma \ref{set to edge lemma}, by condition 2 in Definition \ref{3 prime system definition}, for every set $S \in \mathcal{F}$ there are at least $m-4$ edges $e$ with $|e \cap S|=1$, so there are at least $(m-4)s$ such pairs. Hence $\sum_{v \in H} d(v)s(v) \geq (m-4)s$. \newline

For each vertex $v \in H$, define $\Gamma'(v)$ and $\mathcal{F}_v$ as in the first part of Observation \ref{induction observation}. Then by the first part of Observation \ref{induction observation} and Theorem \ref{many cases}, $s(v)=|\mathcal{F}_v| \leq s_{3,3}'\left(|\Gamma'(v)|\right) = m-d(v)+O(1)$, or equivalently, $d(v)+s(v) \leq m+O(1)$. It is easy to check that $x+9y \geq 16xy$ for all real numbers $x,y \geq 0$ with $x+y \leq 1$. Taking $x=s(v)/[m+O(1)]$ and $y=d(v)/[m+O(1)]$ and summing over all $v \in H$ gives 

$$\frac{4s+18e}{m+O(1)} \ \geq \ \frac{16 \sum_{v \in H} d(v)s(v)}{[m+O(1)]^2} \ \geq \ \frac{16(m-4)s}{[m+O(1)]^2} \ , $$

which implies $e \geq 2s/3+o(s)$. \newline

We now prove the upper bound. Let $s$ be a large integer. We need to construct a maximal $(3,5)$-system $(H,\mathcal{F})$ with $|\mathcal{F}|=s$ and $e(H) \leq 2s+o(s)$. Let $(H_{4,l}',\mathcal{F}_{4,l}')$ be as in property 5 in Construction \ref{system construction}, where $l$ is the smallest integer with $|\mathcal{F}_{4,l}'| \geq s$. Then $l=\left[2^{1/4}+o(1)\right] s^{1/4}$ by property 4 in Construction \ref{system construction}. Since $(H_{4,l}',\mathcal{F}_{4,l}')$ is a maximal $(3,5)$-system, for every missing edge $e$ in $H_{4,l}'$ such that $H_{4,l}'+e$ is $K_3$-free there is a set in $\mathcal{F}_{4,l}'$ containing $e$. Pick a subfamily $\mathcal{F} \subseteq \mathcal{F}_{4,l}'$ with $\left|\mathcal{F}\right|=s$ such that for every missing edge $e$ in $H_{4,l}'$ such that $H_{4,l}'+e$ is $K_3$-free there is a set in $\mathcal{F}$ containing $e$. This is possible since $s \leq |\mathcal{F}_{4,l}'|$ and the number of missing edges $e$ in $H_{4,l}'$ such that $H_{4,l}'+e$ is $K_3$-free is $O(s^{3/4})$ by property 5 in Construction \ref{system construction}, which is at most $s$, since $s$ is large. Then it is easy to see that, since $(H_{4,l}',\mathcal{F}_{4,l}')$ is a maximal $(3,5)$-system and for every missing edge $e$ in $H_{4,l}'$ such that $H_{4,l}'+e$ is $K_3$-free there is a set in $\mathcal{F}$ containing $e$, $(H_{4,l}',\mathcal{F})$ is also a maximal $(3,5)$-system. Finally, $e(H_{4,l}')=l^4-2l^3+2l^2=2s+o(s)$ by property 3 in Construction \ref{system construction}.

\end{proof}

\vspace{1cm}

We will use the following lemma in the proof of Theorem \ref{connection with edges}.

\begin{lemma}\label{clean up lemma}

Let $r \geq 3$ and $t \geq r-2$ be integers and $(H,\mathcal{F})$ be a maximal $(r,t)$-system with one of the following properties.

\begin{enumerate}

\item There is a vertex $v \in H$ with $d(v)+s(v) \leq t$.
\item There are vertices $v \neq w \in H$ with $s(v)=0=s(w)$ and $\Gamma(v)=\Gamma(w)$.

\end{enumerate}

Then there exists a maximal $(r,t)$-system $(H',\mathcal{F}')$ with $|H'|=|H|-1$, $e(H') \leq e(H)+O(1)$ and $|\mathcal{F}|-O(1) \leq |\mathcal{F}'| \leq |\mathcal{F}|$.
\end{lemma} 

\begin{proof}

Suppose first that property 1 holds. Let $H''=H \setminus \{v\}$ and $\mathcal{F}'=\{S \in \mathcal{F} : v \not \in S\}$. Then it is easy to see that $(H'',\mathcal{F}')$ is an $(r,t)$-system, so we can add edges to $H''$ to obtain a graph $H'$ such that $(H',\mathcal{F}')$ is a maximal $(r,t)$-system. Note that $|H'|=|H|-1$ and $|\mathcal{F}|-t \leq |\mathcal{F}'| \leq |\mathcal{F}|$. It remains to show that $e(H') \leq e(H)+O(1)$. Clearly, $e(H'') \leq e(H)$, so it suffices to show that only $O(1)$ edges were added to $H''$ when obtaining $H'$. \newline

Let $e$ be such an added edge. Then, since $(H,\mathcal{F})$ is a maximal $(r,t)$-system, either $H+e$ contains $K_r$ or $(H+e)[S]$ contains $K_{r-1}$ for some $S \in \mathcal{F}$. But $H$ and $H'$ are $K_r$-free and $H[S]$ and $H'[T]$ are $K_{r-1}$-free for all $S \in \mathcal{F}$ and $T \in \mathcal{F}'$, since $(H,\mathcal{F})$ and $(H',\mathcal{F}')$ are $(r,t)$-systems, so either $v$ and $e$ belong to a copy of $K_r$ in $H+e$ or $e$ belongs to a copy of $K_{r-1}$ in $(H+e)[S]$ for some $S \in \mathcal{F}$ with $v \in S$. In particular, either $e \subseteq \Gamma(v)$ or $e \subseteq S$ for some $S \in \mathcal{F}$ with $v \in S$. It follows that there are at most $(t+1)\binom{t}{2}$ such added edges $e$. \newline

Now suppose property 2 holds. Let $H'=H \setminus \{v\}$. Then it is easy to check that $(H',\mathcal{F})$ is a maximal $(r,t)$-system with $|H'|=|H|-1$ and $e(H') \leq e(H)$.

\end{proof}

\vspace{1cm}

We now prove Theorem \ref{connection with edges}, which we will use in the proofs of parts 1 and 2 of Theorem \ref{twin free min deg result}.

\begin{repth15}

For all integers $r \geq 3$ and $t \geq r+3$, 

$$tn + \Omega\left(e_{r,t}'\left[n+o(n)\right]\right) \leq \tsat(n,K_r,t) \leq tn + O\left(e_{r,t}'\left[n+o(n)\right]\right)  \ . $$

\end{repth15}

\begin{proof}

We first prove the lower bound. The proof will be similar to that of Theorem 2 in \cite{Mine}. Let $G$ be a $K_r$-saturated graph on $n$ vertices with $\delta(G) \geq t$ in which every pair of twins is adjacent to a vertex of degree $t$. We need to show that $e(G) \geq tn + \Omega\left(e_{r,t}'\left[n+o(n)\right]\right)$. \newline

By Lemma 4 in \cite{Mine}, there is a set $S \subseteq V(G)$ with $|S|=O(1)$ such that for all $v \in G$, either $|\Gamma(w) \cap S| > t$ for all $w \in \Gamma(v) \setminus S$ or

$$ |\Gamma(v) \cap S| + \frac{1}{2} |\Gamma(v) \setminus S| > t \ . $$

Let $T=S \cup \{v \in G : |\Gamma(v) \cap S| > t \}$. Then

$$ e(T) \ \geq \ e(S,T \setminus S) \ \geq \ (t+1)|T \setminus S| \ = \ (t+1)|T|+O(1) $$

by the definition of $T$ (where $e(T)=e(G[T])$ and $e(S,T \setminus S)$ is the number of edges between $S$ and $T \setminus S$). \newline

We now claim that for all $v \in G$ we have

$$ |\Gamma(v) \cap T| + \frac{1}{2} |\Gamma(v) \setminus T| \geq t \ , $$

with equality if and only if $\Gamma(v) \subseteq T$ and $d(v)=t$. Indeed, if  $|\Gamma(w) \cap S| > t$ for all $w \in \Gamma(v) \setminus S$, then $\Gamma(v) \subseteq T$, so 

$$ |\Gamma(v) \cap T| + \frac{1}{2} |\Gamma(v) \setminus T| = d(v) \geq t \ , $$

since $\delta(G) \geq t$, and if 

$$ |\Gamma(v) \cap S| + \frac{1}{2} |\Gamma(v) \setminus S| > t \ , $$

then

$$  |\Gamma(v) \cap T| + \frac{1}{2} |\Gamma(v) \setminus T| > t \ , $$

since $S \subseteq T$. \newline

Let $ R = \{ v \in G : \Gamma(v) \subseteq T, d(v)=t \} $. Note that $R \setminus T$ is an independent set of vertices of degree $t$. Let $H=G \setminus (R \setminus T)$ and for each subset $Q \subseteq R \setminus T$, let $H_Q=G\left[V(H) \cup Q\right]$. Then we have $e(G)=tn+e(H_Q)-t|H_Q|$ and

$$ e(H_Q) \ = \ e(T) + \sum_{v \in H_Q \setminus T} \left( |\Gamma(v) \cap T| + \frac{1}{2} \left|\Gamma(v) \cap V(H_Q) \setminus T\right| \right)  $$ 

$$ = \ e(T) + \sum_{v \in H_Q \setminus T} \left( |\Gamma(v) \cap T| + \frac{1}{2} \left|\Gamma(v) \setminus T \right| \right) $$

$$ \geq \ (t+1) |T| + O(1)  +  t|H_Q \setminus T| + \frac{1}{2} |H_Q \setminus \left(T \cup R \right)| \ \geq \ \left(t+ \frac{1}{2}\right) |H_Q| - \frac{1}{2} |Q| + O(1) \ .  $$

\vspace{1cm}

Taking $Q=\emptyset$, we obtain $e(G)=tn+e(H)-t|H| \geq tn+|H|/2+O(1)$, so if $|H|=\Omega(n)$, we have $e(G) \geq tn+\Omega(n) \geq tn+\omega\left(e'_{r,t}(n)\right)$, since by part 3 of Theorem \ref{edge estimates}, $e'_{r,t}(n) \ll n^{4/5}$. So suppose $|H|=o(n)$. Take $Q$ to be the set of vertices in $R \setminus T$ which have a twin. Then every vertex in $Q$ is adjacent to a vertex in $H$ of degree $t$, so $|Q| \leq t|H|=o(n)$. Let $\mathcal{F}=\{\Gamma(v) : v \not \in H_Q\}$. Then $(H_Q,\mathcal{F})$ is a maximal $(r,t)$-system by Observation \ref{obs 1}, so $e(H_Q) \geq e_{r,t}'\left[n+o(n)\right]$. Also, $|Q| \leq t|H|=t|H_Q|-t|Q|$ and $e(H_Q) \geq \left(t+1/2\right)|H_Q|-|Q|/2+O(1)$, so

$$t|H_Q| \ \leq \ \left(1-\frac{1}{2t^2+2t+1}\right)e(H_Q) + O(1) \ . $$ 

Hence

$$e(G) \ = \ tn+e(H_Q)-t|H_Q| \ \geq \ tn+ \frac{e(H_Q)}{2t^2+2t+1} + O(1) \ \geq \  tn + \Omega\left(e_{r,t}'\left[n+o(n)\right]\right) \ , $$ 

since by part 3 of Theorem \ref{edge estimates}, $e_{r,t}'\left[n+o(n)\right]=\omega(1)$. \newline

We now prove the upper bound. Let $n$ be a large integer. We need to construct a $K_r$-saturated graph $G$ on $n$ vertices with $\delta(G) \geq t$ and $e(G) \leq tn + O\left(e_{r,t}'\left[n+o(n)\right]\right)$ in which every pair of twins is adjacent to a vertex of degree $t$. Let $C$ be a constant, large enough for the proof to go through, and $N$ be the smallest integer with $N \geq n+Ce_{r,t}'(N)$. Such integers do exist, since by part 3 of Theorem \ref{edge estimates}, $e'_{r,t}(N) \ll N^{4/5}$. Then $N-1<n+Ce_{r,t}'(N-1)$, so $N=n+o(n)$. \newline

Take a maximal $(r,t)$-system $(H,\mathcal{F})$ with $|\mathcal{F}|=N$ and $e(H)=e'_{r,t}(N)$. Starting with $(H,\mathcal{F})$, repeatedly apply Lemma \ref{clean up lemma} until neither property 1 nor property 2 hold. Let $(H',\mathcal{F}')$ be the resulting maximal $(r,t)$-system. Then $d(v)+s(v) > t$ for all $v \in H'$ and either $s(v) \geq 1$ or $s(w) \geq 1$ for all $v \neq w \in H'$ with $\Gamma(v)=\Gamma(w)$. Let $m=|H|$. Then we applied Lemma \ref{clean up lemma} at most $m$ times. By Lemma \ref{set to edge lemma}, $m \ll e'_{r,t}(N)$, so $e(H') \ll e'_{r,t}(N)$ and 

$$ N \ \geq \ |\mathcal{F}'| \ = \ N-O(e'_{r,t}(N)) \ \geq \ n + Ce'_{r,t}(N) - O(e'_{r,t}(N)) \ \geq \ n \ , $$ 

provided $C$ is large enough. Let $m'=\left|H'\right|$. Then

$$ m' \ \leq \ m \ \ll \ e'_{r,t}(N) \ \ll \ N^{4/5} \ \asymp \ n^{4/5} \ , $$

so $(t+1)m' \leq n-m'$, since $n$ is large. \newline

Pick a subfamily $\mathcal{F}'' \subseteq \mathcal{F}'$ with $|\mathcal{F}''|=n-m'$ such that $s'(v) \geq \min\left(s(v),t+1\right)$ for all $v \in H'$, where $s'(v)=|\{S \in \mathcal{F}'' : v \in S\}|$. This is possible since $(t+1)m' \leq n-m' \leq n \leq \left|\mathcal{F}'\right|$. Then $(H',\mathcal{F}'')$ is an $(r,t)$-system and we still have $d(v)+s'(v) > t$ for all $v \in H'$ and either $s'(v) \geq 1$ or $s'(w) \geq 1$ for all $v \neq w \in H'$ with $\Gamma(v)=\Gamma(w)$. \newline

Add edges to $H'$ to obtain a graph $H''$ such that $(H'',\mathcal{F}'')$ is a maximal $(r,t)$-system. We claim that every added edge $e$ is contained in a set in $\mathcal{F}' \setminus \mathcal{F}''$. Indeed, since $(H',\mathcal{F}')$ is a maximal $(r,t)$-system, either $H'+e$ contains $K_r$ or $(H'+e)[S]$ contains $K_{r-1}$ for some $S \in \mathcal{F}'$. But $H''$ is $K_r$-free and $H'[S]$ and $H''[T]$ are $K_{r-1}$-free for all $S \in \mathcal{F}'$ and $T \in \mathcal{F}''$, since $(H',\mathcal{F}')$ and $(H'',\mathcal{F}'')$ are $(r,t)$-systems, so $e$ must belong to a copy of $K_{r-1}$ in $(H'+e)[S]$ for some $S \in \mathcal{F}' \setminus \mathcal{F}''$. In particular, $e \subseteq S$. \newline

It follows that  for all added edges $e$ and vertices $v \in e$, $s(v) \geq 1$ and so $s'(v) \geq \min\left(s(v),t+1\right) \geq 1$. Equivalently, $s'(v) \geq 1$ for all $v \in H'$ with $\Gamma(v) \subset \Gamma'(v)$, where $\Gamma'(v)$ is the neighbourhood of $v$ in $H''$. Hence, we still have either $s'(v) \geq 1$ or $s'(w) \geq 1$ for all $v \neq w \in H'$ with $\Gamma'(v)=\Gamma'(w)$. We also still have $d'(v)+s'(v) > t$ for all $v \in H'$, where $d'(v)$ is the degree of $v$ in $H''$, since $d'(v) \geq d(v)$. Let $N' \in \{N-1,N\}$ be such that $\max[e_{r,t}'(N-1),e_{r,t}'(N)]=e'_{r,t}(N')$. Note that $N'=n+o(n)$. Since 

$$|\mathcal{F}' \setminus \mathcal{F}''| \ \leq \ N-n+m' \ < \ Ce_{r,t}'(N-1)+m+1 \ \ll \ e'_{r,t}(N') \ , $$

it also follows from the previous paragraph that the number of added edges is $O(e'_{r,t}(N'))$, so $e(H'') \ll e'_{r,t}(N')$.  \newline

Let $G=G(H'',\mathcal{F}'')$. Then $G$ is $K_r$-saturated by Observation \ref{obs 1} and $|G|=|H''|+\left|\mathcal{F}''\right|=n$. For every vertex $v \in H''$, the degree of $v$ in $G$ is $d'(v)+s'(v)>t$, and all other vertices in $G$ have degree $t$. Hence $\delta(G) \geq t$. Let $v \neq w \in G$ be twins. We need to show that they are adjacent to a vertex of degree $t$, or equivalently, that $v, w \in H''$ and $s'(v)=s'(w) \geq 1$. Since $v$ and $w$ have the same degree, either $v,w \not \in H''$ or $v,w \in H''$. The former is impossible, since $\mathcal{F}''$ has no repeated elements, and if the latter holds, we must have $s'(v)=s'(w) \geq 1$, since $s'(v)=s'(w)$ and $\Gamma'(v) = \Gamma'(w)$. Finally, $e(G)=e\left(H''\right)+t\left|\mathcal{F}''\right|=tn + O(e'_{r,t}(N'))$. \newline

\end{proof}

\vspace{1cm}

Next, we prove Theorem \ref{twin free result}.

\begin{repth12}
We have the following estimates for $\tsat(n,K_r)$.
\begin{enumerate}
\item For all integers $r \geq 3$,
$$\left(r+2\right)n + o(n) \leq \tsat(n,K_r) \leq (r+3)n+o(n) \ .$$
\item In the special case $r=3$, we obtain a stronger lower bound.
$$\left(5 + \frac{2}{3}\right)n + o(n) \leq \tsat(n,K_3) \leq 6n+o(n) \ .$$
\end{enumerate}
\end{repth12}

\begin{proof}

We first prove the lower bounds. Let $G$ be a twin-free $K_r$-saturated graph on $n$ vertices. We need to show that $e(G) \geq \left(r+2\right)n + o(n)$ and $e(G) \geq \left(5 + 2/3 \right)n + o(n)$ in the special case $r=3$. If $e(G)=\omega(n)$, there is nothing to prove, so we may assume $e(G)=O(n)$. Then by Theorem 4 in \cite{Mine}, $G$ has a vertex cover $C$ of size $m=O(n/\log n)$. Let $H=G[C]$ and $\mathcal{F}=\{\Gamma(v) : v \not \in C\}$. Then $(H,\mathcal{F})$ is a maximal $r$-system by Observation \ref{obs 1}. Note that $\mathcal{F}$ has no repeated elements, since $G$ is twin-free. \newline

Recall (see the second paragraph of the introduction) that $r-2$ is the smallest possible degree of a vertex in a $K_r$-saturated graph on at least $r-1$ vertices. Since $n$ is large, it follows that every set in $\mathcal{F}$ has size at least $r-2$. For each integer $r-2 \leq t \leq r+1$, let $\mathcal{F}_t=\{S \in \mathcal{F} : |S|=t \}$. Then, since $(H,\mathcal{F})$ is a maximal $r$-system, $(H,\mathcal{F}_t)$ is an $(r,t)$-system, so $|\mathcal{F}_t| \leq s_{r,t}(m) \ll m \ll n/ \log n$ by parts 1 through 4 of Theorem \ref{set estimates}. Hence there are only $O(n/\log n)$ vertices outside $C$ of degree at most $r+1$, so $e(G) \geq \left(r+2\right)n + O(n/\log n)$. \newline

In the special case $r=3$, let $\mathcal{F}_5=\{S \in \mathcal{F} : |S|=5 \}$. Then, as before, since $(H,\mathcal{F})$ is a maximal $3$-system, $(H,\mathcal{F}_5)$ is a $(3,5)$-system. Hence $e(H) \geq 2|\mathcal{F}_5|/3 + o\left(|\mathcal{F}_5|\right)$ by part 5 of Theorem \ref{edge estimates}. Also, since there are only $O(n/\log n)$ vertices outside $C$ of degree at most $4$, $e\left(C,V(G)\setminus C \right) \geq 6n-|\mathcal{F}_5|+O(n/ \log n)$. So 

$$ e(G) \ = \ e(H) + e\left(C,V(G)\setminus C \right) \ \geq \ 6n-|\mathcal{F}_5|/3 + o\left(|\mathcal{F}_5|\right) + O(n/ \log n) \ . $$

But $|\mathcal{F}_5| \leq n$, so $e(G) \geq \left(5 + 2/3 \right)n + o(n)$. \newline

We now prove the upper bounds. By Lemma \ref{twin-free-conical}, it is sufficient to prove the upper bound in part 2. Let $n$ be a large integer. We need to construct a twin-free $K_3$-saturated graph $G$ on $n$ vertices with $e(G) \leq 6n+o(n)$. Let $l=\left\lceil n^{1/5} \right\rceil$ and $(H_{5,l}',\mathcal{F}_{5,l}')$ be as in property 5 in Construction \ref{system construction}. By properties 2 and 4 in Construction \ref{system construction}, $|H_{5,l}'|=5l^2+1$ and $|\mathcal{F}_{5,l}'|=l^5$. Pick a subfamily $\mathcal{F} \subseteq \mathcal{F}_{5,l}'$ with $\left|\mathcal{F}\right|=n-5l^2-1$ such that every missing edge in $H_{5,l}'$ that is contained in a set in $\mathcal{F}_{5,l}'$ is contained in a set in $\mathcal{F}$. This is possible since $n-5l^2-1 \leq l^5=|\mathcal{F}_{5,l}'|$ and the number of missing edges in $H_{5,l}'$ is $O(n^{4/5})$, which is at most $n-5l^2-1=n+O\left(n^{2/5}\right)$, since $n$ is large. Then it is easy to see that, since $(H_{5,l}',\mathcal{F}_{5,l}')$ is a maximal $(3,6)$-system, so is $(H_{5,l}',\mathcal{F})$. \newline

Let $G=G(H_{5,l}',\mathcal{F})$. Then $G$ is $K_3$-saturated by Observation \ref{obs 1} and $|G|=|H_{5,l}'|+|\mathcal{F}|=n$. The vertices in $H_{5,l}'$ have degree $\omega(1)$ in $G$ by property 3 in Construction \ref{system construction} and all other vertices in $G$ have degree $6$. It follows that $\delta(G)=6$. We will use this to prove the upper bound in part 4 of Theorem \ref{twin free min deg result}. It also follows that for all twins $v \neq w \in G$, either $v,w \not \in H_{5,l}'$ or $v,w \in H_{5,l}'$. But $\mathcal{F}$ has no repeated elements and $H_{5,l}'$ is twin-free by property 5 in Construction \ref{system construction}, so $G$ is twin-free. Finally, $e(G)=e(H_{5,l}')+6|\mathcal{F}|=6n+O(n^{4/5})$. \newline     

\end{proof}

\vspace{1cm}

Finally, we prove Theorem \ref{twin free min deg result}.

\begin{repth16}
We have the following estimates for $\tsat(n,K_r,t)$.
\begin{enumerate}
\item For all integers $r \geq 3$ and $t \geq r+3$, 
$$tn + \Omega\left(n^{1/(t-r+2)} \right) \leq \tsat(n,K_r,t) \leq tn+ O\left( n^{4/(t-r+2) } \right)  \ . $$
\item For all integers $t \geq 6$, 
$$tn + \Omega\left(n^{2/(t-3)} \right) \leq \tsat(n,K_3,t) \leq tn + O\left(n^{4/(t-1)}\right)  \ . $$
\item For all integers $r \geq 3$ and $r-2 \leq t \leq r+2$,
$$\left(r+2\right)n + o(n) \leq \tsat(n,K_r,t) \leq (r+3)n+o(n) \ .$$
\item For all integers $1 \leq t \leq 5$,
$$\left(5 + \frac{2}{3}\right)n + o(n) \leq \tsat(n,K_3,t) \leq 6n+o(n) \ .$$
\end{enumerate}
\end{repth16}

\begin{proof}

Parts 1 and 2 follow from Theorem \ref{connection with edges} and parts 3 and 4 of Theorem \ref{edge estimates}, respectively. We now prove parts 3 and 4. We first prove the lower bounds. The proof will be similar to that of the lower bounds in Theorem \ref{twin free result}. Let $G$ be a $K_r$-saturated graph on $n$ vertices with $\delta(G) \geq t$ in which every pair of twins is adjacent to a vertex of degree $t$. We need to show that $e(G) \geq \left(r+2\right)n + o(n)$ and $e(G) \geq \left(5 + 2/3 \right)n + o(n)$ in the special case $r=3$. \newline

If $e(G)=\omega(n)$, there is nothing to prove, so we may assume $e(G)=O(n)$. Then, as in the proof of the lower bounds in Theorem \ref{twin free result}, by Theorem 4 in \cite{Mine}, $G$ has a vertex cover $C$ of size $O(n/\log n)$. Let $S$ be the set of vertices $v \not \in C$ which have a twin. Then every vertex in $S$ is adjacent to a vertex in $C$ of degree $t$, so $|S|\leq t|C|=O(n/\log n)$. Let $H=G\left[C \cup S\right]$, $\mathcal{F}=\{\Gamma(v) : v \not \in C \cup S\}$ and $m=|H|=O(n/\log n)$. Then $(H,\mathcal{F})$ is a maximal $r$-system by Observation \ref{obs 1}. Note that $\mathcal{F}$ has no repeated elements.   \newline

Recall again (see the second paragraph of the introduction) that $r-2$ is the smallest possible degree of a vertex in a $K_r$-saturated graph on at least $r-1$ vertices. Since $n$ is large, it follows that every set in $\mathcal{F}$ has size at least $r-2$. For each integer $r-2 \leq t \leq r+1$, let $\mathcal{F}_t=\{S \in \mathcal{F} : |S|=t \}$. Then, since $(H,\mathcal{F})$ is a maximal $r$-system, $(H,\mathcal{F}_t)$ is an $(r,t)$-system, so $|\mathcal{F}_t| \leq s_{r,t}(m) \ll m \ll n/ \log n$ by parts 1 through 4 of Theorem \ref{set estimates}, as before. Hence there are only $O(n/\log n)$ vertices outside $C \cup S$ of degree at most $r+1$, so $e(G) \geq \left(r+2\right)n + O(n/\log n)$. \newline

In the special case $r=3$, let $\mathcal{F}_5=\{S \in \mathcal{F} : |S|=5 \}$. Then, as before, since $(H,\mathcal{F})$ is a maximal $3$-system, $(H,\mathcal{F}_5)$ is a $(3,5)$-system. Hence $e(H) \geq 2|\mathcal{F}_5|/3 + o\left(|\mathcal{F}_5|\right)$ by part 5 of Theorem \ref{edge estimates}. Also, since there are only $O(n/\log n)$ vertices outside $C \cup S$ of degree at most $4$, $e\left(C,V(G)\setminus C \right) \geq 6n-|\mathcal{F}_5|+O(n/ \log n)$. So 

$$ e(G) \ = \ e(H) + e\left(C,V(G)\setminus C \right) \ \geq \ 6n-|\mathcal{F}_5|/3 + o\left(|\mathcal{F}_5|\right) + O(n/ \log n) \ . $$

But $|\mathcal{F}_5| \leq n$, so $e(G) \geq \left(5 + 2/3 \right)n + o(n)$, as in the proof of the lower bounds in Theorem \ref{twin free result}. \newline

Finally, we prove the upper bounds. By Lemma \ref{twin-free-conical-min-degree}, it is sufficient to prove the upper bound in part 4. Recall that in the proof of the upper bound in part 2 of Theorem \ref{twin free result}, we constructed a twin-free $K_3$-saturated graph $G$ on $n$ vertices with $\delta(G)=6$ and $e(G)=6n+O(n^{4/5})$ for every large integer $n$. The upper bound in part 4 follows. \newline

\end{proof}

\section{Open problems}\label{Open problems}

The obvious open problems are to improve the estimates for the extremal quantities we defined.

\begin{probl}
Improve the estimates for $\tsat(n,K_r)$ in Theorem \ref{twin free result}.
\end{probl}

\begin{probl}
Improve the estimates for $\tsat(n,K_r,t)$ in Theorem \ref{twin free min deg result}.
\end{probl}

\begin{probl}
Improve the estimates for $s_{r,t}(m)$ in Theorem \ref{set estimates}.
\end{probl}

\begin{probl}
Improve the estimates for $e_{r,t}(s)$ and $e'_{r,t}(s)$ in Theorem \ref{edge estimates}.
\end{probl}

\vspace{0.5cm}

In light of Theorem \ref{existence}, one might ask the following.

\begin{q}
For which $r$, $t$ and $n$ does there exist a $K_r$-saturated graph $G$ on $n$ vertices with $\delta(G) \geq t$ that cannot be obtained by blowing up a smaller $K_r$-saturated graph $H$ with $\delta(H) \geq t$?
\end{q}

\vspace{0.5cm}

Though in some sense the $K_r$ case is the most natural one, one could also define $\tsat(n,H)$ to be the minimum number of edges in a twin-free $H$-saturated graph on $n$ vertices for other graphs $H$. As in the case $H=K_r$, it is not immediately clear whether such graphs even exist. By Theorem \ref{existence}, in the case $H=K_r$, such graphs do exist for fixed $r$ and large $n$.

\begin{q}\label{Q}
For which graphs $H$ do there exist twin-free $H$-saturated graphs on $n$ vertices for large enough $n$?
\end{q}

The answer to Question \ref{Q} is not ``All graphs". Indeed, it is easy to show that when $H$ is a matching, all $H$-free graphs are blow-ups of a graph with $O(1)$ vertices. Indeed, let $G$ be an $H$-free graph. Let $S$ be the set of matched vertices in a maximal matching in $G$. Then, since $G$ is $H$-free, $|S|=O(1)$, and since the matching is maximal, $V(G) \setminus S$ is an independent set. Together, these imply that $G$ is a blow-up of a graph with $O(1)$ vertices. \newline

K\'aszonyi and Tuza showed that $\sat(n,H)=O(n)$ for all graphs $H$ (see Theorem 1 in \cite{Kaszonyi and Tuza}). It is not at all clear whether the same should hold for $\tsat(n,H)$, when it is well-defined.

\begin{q}
Do we have  $\tsat(n,H)=O(n)$ for every graph $H$ for which there exist twin-free $H$-saturated graphs on $n$ vertices for large enough $n$?
\end{q}

Tuza conjectured that the ratio $\sat(n,H)/n$ converges for all graphs $H$ (see Conjecture 10 in \cite{Tuza's Conjecture}). Tuza's Conjecture is still wide open. It is again not at all clear whether the same should hold for $\tsat(n,H)$ when $\tsat(n,H)=O(n)$.

\begin{q}
Does $\tsat(n,H)/n$ converge for all graphs $H$ for which $\tsat(n,H)=O(n)$?
\end{q}

\end{document}